\documentclass[11pt]{article}

\usepackage[pagebackref,colorlinks=true,pdfpagemode=none,urlcolor=blue,
linkcolor=blue,citecolor=blue]{hyperref}

\usepackage{amsmath,amsfonts,amssymb,amsthm}
\usepackage{mathrsfs}
\usepackage{color}

\newtheorem{theorem}{Theorem}[section]
\newtheorem{lemma}[theorem]{Lemma}
\newtheorem{definition}[theorem]{Definition}

\newtheorem{assumption}[theorem]{Assumption}
\newtheorem{proposition}[theorem]{Proposition}

\theoremstyle{remark}
\newtheorem{remark}[theorem]{Remark}

\renewenvironment{proof}[1][Proof]{ {\itshape \noindent {#1.}} }{$\Box$
\medskip}

\numberwithin{equation}{section}
\newcommand{\R}{\mathbb{R}}

\newcommand{\Pb}{\mathbb{P}}
\newcommand{\E}{\mathbb{E}}

\newcommand{\C}{\mathcal{C}}
\newcommand{\X}{\mathcal{X}}

\newcommand{\V}{\mathbb{V}}

\newcommand{\RR}{\mathcal{R}}
\newcommand{\eps}{\varepsilon}
\def\les{\lesssim}

\begin{document}

\title{Weak Convergence Approach for Parabolic Equations\\ with Large, Highly Oscillatory, Random Potential}
\author{Yu Gu\thanks{Department of Applied Physics \& Applied
Mathematics, Columbia University, New York, NY 10027 (yg2254@columbia.edu; gb2030@columbia.edu)}  \and Guillaume Bal\footnotemark[1]}

\maketitle

\begin{abstract}
This paper concerns the macroscopic behavior of solutions to parabolic equations with large, highly oscillatory, random potential. When the correlation function of the random potential satisfies a specific integrability condition, we show that the random solution converges, as the correlation length of the medium tends to zero, to the deterministic solution of a homogenized equation in dimension $d\geq3$. Our derivation is based on a Feynman-Kac probabilistic representation and the Kipnis-Varadhan method applied to weak convergence of Brownian motions in random sceneries. For sufficiently mixing coefficients, we also provide an optimal rate of convergence to the homogenized limit using  a quantitative martingale central limit theorem. As soon as the above integrability condition fails, the solution is expected to remain stochastic in the limit of a vanishing correlation length. For a large class of potentials given as functionals of Gaussian fields, we show the convergence of solutions to stochastic partial differential equations (SPDE) with multiplicative noise. The Feynman-Kac representation and the corresponding weak convergence of Brownian motions in random sceneries allows us to explain the transition from deterministic to stochastic limits as a function of the correlation function of the random potential.
\end{abstract}

\section{Introduction}

Solutions of partial differential equations with small scale structures arise in many aspects of physical and applied sciences. Homogenization theory has proved to be useful, both from the theoretical and numerical points of view, to provide macroscopic descriptions for such solutions. We consider here the setting of a parabolic equation with a large and highly oscillatory random potential. One of the salient features of such models is that the properties of the limiting macroscopic model strongly depend on the correlation properties of the random medium. When an integrability condition on the correlation function is met, then the stochastic solution converges in the limit of vanishing correlation length to a deterministic, homogenized solution. However, when that condition is not satisfied, the random solution remains stochastic in that limit and converges to the solution of a stochastic partial differential equation (SPDE) with multiplicative noise. The main objective of this paper is to provide a derivation of such results and an understanding of the transition from deterministic to stochastic limits from a probabilistic point of view. When the solution converges to a deterministic limit, we also derive optimal rates of convergence provided the potential satisfies certain mixing conditions.

Similar such equations have been analyzed recently. When the random potential is Gaussian, a Duhamel infinite series expansions and combinatorial techniques allows us to understand such a convergence for a relatively large class of parabolic equations including parabolic Anderson and Schr\"odinger models; see \cite{bal2009convergence,bal2010homogenization,B-AMRX-11,ZB-CMS-13,ZB-SD-13}. These explicit methods do not seem to extend to non-Gaussian potentials. In the one-dimensional setting of the heat equation, the convergence to a stochastic limit in the mixing case was addressed in \cite{pardoux2006homogenization} using the same probabilistic (Feynman-Kac formula) representation we consider in this paper. The convergence to deterministic limits for time-dependent potentials (not considered in this paper) has been considered in \cite{pardoux2012homogenization,hairer2013random}.

In this paper, we adapt the Feynman-Kac approach to analyze a parabolic equation in dimension $d\geq 3$ of the form 
\begin{equation}
\partial_t u_\eps=\frac12\Delta u_\eps+iV_\eps u_\eps,
\label{eq:mainEq}
\end{equation} where $V_\eps(x)=\eps^{-\gamma}V(x/\eps)$ is a large, time-independent, highly oscillatory, random potential. An imaginary potential is introduced to obtain a uniform bound on the energy of the solution, which considerably simplifies the analysis of exponential functionals of Brownian motion and the passage to the limit as $\eps\to0$. The corresponding heat equation (with $iV_\eps$ replaced by $V_\eps$) might be analyzed using techniques developed in \cite{hairer2013random} but this problem is not considered further here. 
Note that the scalar equation of the form $\partial_t u=\frac12\Delta u+i Vu$ may be recast as the system
\begin{equation}
\partial_t\left(\begin{array}{c}
u_1\\
u_2
\end{array}\right)=\left(\begin{array}{cc}
\frac12\Delta & 0 \\
0 & \frac12\Delta
\end{array}\right)\left(\begin{array}{c}
u_1\\
u_2
\end{array}\right)+\left(\begin{array}{cc}
0 & -V \\
V & 0
\end{array}\right)\left(\begin{array}{c}
u_1\\
u_2
\end{array}\right),
\end{equation}
with $u=u_1+i u_2$. Here, $V$ may model a conservative process of  interaction between two components otherwise satisfying independent diffusions. We obtain \eqref{eq:mainEq} by looking at the long time, large distance asymptotic limit $u_\eps(t,x)=u(t/\eps^2,x/\eps)$ for $\eps\ll 1$. To obtain nontrivial effects from the potential, it suffices to impose on $V$ a weak amplitude $\eps^\zeta$ with $\zeta>0$ to be determined. Deriving the equation of $u_\eps(t,x)$ leads to \eqref{eq:mainEq} with $\gamma=2-\zeta$. We analyze the asymptotic behavior of $u_\eps$  as $\eps\to 0$, and prove homogenization and convergence to SPDE under different assumptions on $V(x)$.


There is a large body of literature on stochastic homogenization, starting from the work of Kozlov \cite{kozlov1979averaging} and Papanicolaou-Varadhan \cite{papanicolaou1979boundary}, where elliptic operators of the form $\nabla\cdot a(\frac{.}{\eps})\nabla$ are considered for stationary and ergodic coefficients.  Homogenization results show that as $\eps\to 0$ the elliptic operator converges in an appropriate sense to a homogenized operator with constant coefficients. The rates of convergence are less well understood. Yurinskii \cite{yurinskii1986averaging} provided the first quantitative estimates for the statistical error. Discrete cases have been analyzed in \cite{conlon2000homogenization,gloria2011optimal,marahrens2013annealed, mourrat2012kantorovich}, using analytic and probabilistic approaches respectively. For the fully-nonlinear case, see \cite{caffarelli2010rates,armstrong2013quantitative}.
When $d=1$, an explicit solution is available, which simplifies the analysis of the statistical fluctuations and allows to derive central limits for the random fluctuations, see \cite{bourgeat1999estimates,bal2008random, gu2012random}. In the setting of bounded random potentials, \cite{figari1982mean, bal2008central, bal2011corrector, bal2012corrector} analyzed elliptic equations and derived central limit results. 

From a probabilistic point of view, different realizations of the random differential operator $\nabla\cdot a(\frac{.}{\eps})\nabla$ correspond to families of diffusion processes, so that homogenization may be recast as a problem of  weak convergence of random motions in random environments; see \cite{komorowski2012fluctuations} and the references therein. For the heat equation considered here, our setting is that of a Brownian motion propagating in random sceneries. It is the continuous counterpart of Kesten's model of random walk in random scenery, for which the invariance principle has been proved in \cite{kesten1979limit, bolthausen1989central}.  The weak convergence of Brownian motion  in random scenery is based on the Kipnis-Varadhan approach  \cite{kipnis1986central}. We apply to the homogenization setting the point of view of the medium seen from an observer and their methods of corrector equation and martingale decomposition.  The same probabilistic approach was used in \cite{lejay2001homogenization} to handle equations in more general forms with random potentials written as derivatives of bounded processes.

Theorem \ref{thm:THhomo} below provides a convergence result for the most general class of potentials for which such a convergence is expected; see Assumption \ref{ass:finiteVar} below. Using the probabilistic representation, the difference between the heterogeneous and homogenized solutions is approximately reduced to the Wasserstein distance between martingales and Brownian motions. We use the quantitative martingale central limit theorem developed in \cite{mourrat2012kantorovich} to estimate the Wasserstein distance and obtain the optimal convergence rates when the random potential satisfies additional mixing conditions in Theorem \ref{thm:THcorrector}. The mixing property is only used in moment estimation. While this imposes the constraint that the random potentials be sufficiently short-range-correlated,  we apply the same quantitative martingale central limit theorem and extend the result to long-range-correlated Gaussian potentials; see Theorem \ref{thm:THgaussian} below.

When Assumption \ref{ass:finiteVar} below is not satisfied, we do not expect convergence to a deterministic homogenized solution. Exhibiting all possible macroscopic limits in this case seems to be out of reach. From the analysis of the simpler setting of random fluctuations beyond homogenization \cite{bal2008random,bal2012corrector,gu2012random}, we expect the class of possible limits to be rather large. We consider here a large class of random potentials with covariance function decaying sufficiently slowly so that Assumption \ref{ass:finiteVar} is violated and prove a result of convergence to SPDE in Theorem \ref{thm:THspde}. A sharp transition to stochasticity is thus observed beyond Assumption \ref{ass:finiteVar}. 
In the long-range-correlation setting, these results relate to limit theorems of sum of strongly correlated random variables, where non-Gaussian limit might appear in certain circumstances \cite{taqqu1975weak}. Our random coefficients are chosen as functionals of Gaussian processes and we obtain a SPDE driven by multiplicative Gaussian noise in the limit. Similar type of limiting equation is analyzed in \cite{hu2011feynman} by Feynman-Kac formula. In \cite{komorowski2010asymptotic}, the heat equation with long-range correlated Gaussian potential is studied with a similar type of limiting equation as in \cite{hu2011feynman}.

 The rest of paper is organized as follows. We state our main results in Section \ref{sec:result} and discuss possible extensions in Section \ref{sec:conclu}. We then prove convergence to homogenized limit and error estimate under different assumptions in Section \ref{sec:homo}. The result of convergence to SPDE is proved in Section \ref{sec:spde}. We present some technical lemmas in the Appendix. 
 
 Here are notations used throughout the paper. In the product probability space, we use $\E$ to denote the expectation only with respect to random coefficients and $\E_B$ the expectation only with respect to the Brownian motion starting from the origin. Joint expectation is denoted by $\E\E_B$.  $a\les b$ stands for $a\leq Cb$ for some $\eps-$independent constant $C>0$. We use $a\wedge b=\min(a,b)$. $N(\mu,\sigma^2)$ is the Gaussian distribution with mean $\mu$ and variance $\sigma^2$, and $q_t(x)$ denotes the density function of $N(0,t)$. When we write $\Psi(r)\les 1\wedge r^{-\beta}$ for any $\beta>0$, the constant of proportionality might depend on $\beta$.

\section{Main results}
\label{sec:result}

We rely on the Feynman-Kac representation for the solution to \eqref{eq:mainEq} in dimension $d\geq 3$. Assuming the initial condition $u_\eps(0,x)=f(x)$ for $f\in \C_b(\R^d)$, the Feynman-Kac solution is given by
\begin{equation}
u_\eps(t,x)=\E_B\{ f(x+B_t)\exp(i\int_0^t V_\eps(x+B_s)ds)\}.
\label{eq:FKrep}
\end{equation}
Without any regularity assumption on $V_\eps$, \eqref{eq:mainEq} is not always solvable in the classical sense, and the solution given by \eqref{eq:FKrep} is not necessarily a classical solution. In Propositon \ref{prop:wkSolu}, we show it is indeed a weak solution almost surely provided that $V_\eps(x)=\eps^{-\gamma}V(x/\eps)$ for some random potential $V(x)$ that has locally bounded sample path.

Since $V(x)$ may be unbounded, proving uniqueness of the solution to \eqref{eq:mainEq} is a difficult task. Such a task becomes easy when the equation is posed on a bounded domain since $V$ is then bounded almost surely. But calculations with the corresponding Brownian motion on bounded domains involve standard complications which we wish to avoid here. When we refer to "the" solution to \eqref{eq:mainEq}, we therefore mean the weak solution given by the Feynman-Kac probabilistic representation in the rest of the paper.

In the following, we state the main results of homogenization and convergence to SPDE respectively.

\subsection{Convergence to homogenized limit and error estimate}

Let $(\Omega,\mathcal{F},\Pb)$ be a \emph{random medium} associated with a group of measure-preserving, ergodic transformation $\{\tau_x,x\in\R^d\}$. Let $\V\in L^2(\Omega)$ with $\int_\Omega \V(\omega)\Pb(d\omega)=0$. Define $V(x,\omega)=\V(\tau_x\omega)$ and we consider the equation when $d\geq 3$:
\begin{equation}
\partial_t u_\eps(t,x,\omega)=\frac12\Delta u_\eps(t,x,\omega)+i\frac{1}{\eps}V(\frac{x}{\eps},\omega)u_\eps(t,x,\omega),
\label{eq:mainEqhomo}
\end{equation}
with initial condition $u_\eps(0,x,\omega)=f(x)$ for $f\in \C_b(\R^d)$, i.e., in \eqref{eq:mainEq} we choose $\gamma=1$. For detailed setup of random medium, we refer to e.g. \cite{papanicolaou1979boundary, komorowski2012fluctuations}. We will write $u_\eps(t,x)$ and $V(x)$ from now on.

Let $\{D_k,k=1,\ldots,d\}$ be the $L^2(\Omega)$ generator of $T_x$ defined as $T_xf(\omega)=f(\tau_x\omega)$, and Laplacian operator $L=\frac12\sum_{k=1}^d D_k^2$. We use $\langle.,.\rangle$ to denote the inner product in $L^2(\Omega)$ and $\|.\|$ the $L^2(\Omega)$ norm, and assume that 
\begin{assumption}
\begin{equation}
\langle \V, -L^{-1}\V\rangle<\infty.
\end{equation}
\label{ass:finiteVar}
\end{assumption}

By assuming $T_x$ is strongly continuous in $L^2(\Omega)$, we obtain the spectral resolution
\begin{equation}
T_x=\int_{\R^d}e^{i\xi x}U(d\xi),
\end{equation}
where $U(d\xi)$ is the associated projection valued measure. We assume there is a non-negative power spectrum $\hat{R}(\xi)$ associated with $\V$, i.e., $\hat{R}(\xi)d\xi=(2\pi)^d\langle U(d\xi)\V,\V\rangle$. Then Assumption \ref{ass:finiteVar} is equivalent to
\begin{equation}
\int_{\R^d}\frac{\hat{R}(\xi)}{|\xi|^2}d\xi<\infty.
\end{equation}
We also have that
\begin{equation}
R(x):=\langle T_x\V,\V\rangle=\frac{1}{(2\pi)^d}\int_{\R^d}e^{i\xi\cdot x}\hat{R}(\xi)d\xi.
\end{equation}

Defining $$\sigma^2=2\langle \V, -L^{-1}\V\rangle=\frac{4}{(2\pi)^d}\int_{\R^d}\frac{\hat{R}(\xi)}{|\xi|^2}d\xi,$$ and $u_{hom}(t,x)$ such that
\begin{equation}
\partial_t u_{hom}(t,x)=\frac12\Delta u_{hom}(t,x)-\frac12\sigma^2u_{hom}(t,x)
\label{eq:limitHOMO}
\end{equation}
with same initial condition $u_{hom}(0,x)=f(x)$, we have the following theorem:
\begin{theorem}[\emph{Homogenization}]
Under Assumption \ref{ass:finiteVar}, $u_\eps(t,x)\to u_{hom}(t,x)$ in probability as $\eps\to 0$.
\label{thm:THhomo}
\end{theorem}

\begin{remark}
Clearly, Assumption \ref{ass:finiteVar} merely ensures $\sigma^2$, i.e., the homogenized constant, to be well-defined. Since $u_\eps$ and $u_{hom}$ are both bounded, moment convergence holds as well. Furthermore, if $f\in L^1(\R^d)$, $|u_\eps(t,.)|, |u_{hom}(t,.)|$ are both bounded by $\mathcal{U}(t,.)\in L^2(\R^d)$, which solves $\partial_t \mathcal{U}=\frac12\Delta \mathcal{U}$ with initial condition $\mathcal{U}(0,x)=|f(x)|$, so $\int_{\R^d}\E\{|u_\eps(t,x)-u_{hom}(t,x)|^2\}dx\to 0$ as $\eps \to 0$.
\end{remark}

We are also interested in the convergence rate of $u_\eps\to u_{hom}$. To give error estimate, one possible assumption we need is the following strongly mixing property of the random potential $V(x)$:

\begin{assumption}
$\E\{V^6(x)\}<\infty$ and there exists a mixing coefficient $\rho(r)$ decreasing in $r\in [0,\infty)$ such that for any $\beta>0$, $\rho(r)\leq C_\beta (1\wedge r^{-\beta})$ for some $C_\beta>0$ and the following bound holds
\begin{equation}
\E\{\phi_1(V)\phi_2(V)\}\leq \rho(r)\sqrt{\E\{\phi_1^2(V)\}\E\{\phi_2^2(V)\}}
\end{equation}
for any two compact sets $K_1,K_2$ with $d(K_1,K_2)=\inf_{x_1\in K_1,x_2\in K_2}\{|x_1-x_2|\}\geq r$ and any random variables $\phi_1(V),\phi_2(V)$ with $\phi_i(V)$ being $\mathcal{F}_{K_i}-$measurable and $\E\{\phi_i(V)\}=0$.
\label{ass:mixing}
\end{assumption}

\begin{remark}
Under Assumption \ref{ass:mixing}, we have $|R(x)|=|\E\{V(0)V(x)\}|\les  1\wedge |x|^{-\beta}$ for any $\beta>0$. Note that \begin{equation}
\sigma^2=\frac{4}{(2\pi)^d}\int_{\R^d}\frac{\hat{R}(\xi)}{|\xi|^2}d\xi=\frac{1}{\pi^{\frac{d}{2}}}\Gamma(\frac{d}{2}-1)\int_{\R^d}\frac{R(x)}{|x|^{d-2}}dx,
\end{equation}
so the strongly mixing assumption implies finiteness of the homogenization constant.
\label{re:decayRx}
\end{remark}

The following is the result of convergence rate for strongly mixing potentials:
\begin{theorem}[\emph{Error estimate for strongly mixing potentials}]
Under Assumption \ref{ass:mixing}, if $f\in \C_c^\infty(\R^d)$, the following error estimates hold:
\begin{equation}
\E\{|u_\eps(t,x)-u_{hom}(t,x)|\}\leq (1+t)C_{d,f,\rho}\left\{
\begin{array}{ll}
\sqrt{\eps} & d=3,\\
\eps \sqrt{|\log \eps|} & d=4,\\
\eps & d>4.
\end{array} \right.
\end{equation}
\label{thm:THcorrector}
\end{theorem}

\begin{remark}
As suggested by the notation, $C_{d,f,\rho}$ only depends on the dimension, initial condition and mixing coefficient. If we follow the proof, it is easy to check that we only need to assume $\rho(r)\les 1\wedge r^{-\beta}$ for sufficiently large $\beta$, and the regularity assumption on $f$ could be improved as well.
\end{remark}

The error estimate given in Theorem \ref{thm:THcorrector} is universal in the sense that it is independent of the potential as long as Assumption \ref{ass:mixing} holds. The strongly mixing property is only used when estimating moments of $V(x)$ and controlling relevant integrals. For Gaussian potentials, the calculation of moments is straightforward, and this  enables us to extend the error estimate to long-range-correlation setting.

\begin{assumption}
$V(x)$ is a zero-mean Gaussian random field with auto-covariance function $R(x)\sim |x|^{-\beta}$ as $x\to \infty$ for $\beta\in(2,d)$.
\label{ass:gauAss}
\end{assumption}

The condition $\beta>2$ ensures that $R(x)|x|^{2-d}$ is integrable so Assumption \ref{ass:finiteVar} holds. On the other hand, $\beta<d$ so $R(x)$ is not integrable and it is the long-range-correlated case. The following theorem is a precise description of how the homogenization error depends on the interaction between the dimension $d$ and the decay rate $\beta$ of auto-covariance function.

\begin{theorem}[\emph{Error estimate for long-range-correlated Gaussian potentials}]
Under Assumption \ref{ass:gauAss}, if $f\in \C_c^\infty(\R^d)$, the following error estimates hold:
 \begin{itemize}
 
 \item when $d=3,4$, \begin{equation}
 \E\{|u_\eps(t,x)-u_{hom}(t,x)|\}\leq (1+t)C_{d,f,\beta}\eps^{\frac{\beta}{2}-1},
 \end{equation}

\item when $d>4$, 
 \begin{equation}
\E\{|u_\eps(t,x)-u_{hom}(t,x)|\}\leq(1+t)C_{d,f,\beta} \left\{
\begin{array}{ll}
\eps^{\frac{\beta}{2}-1} & \beta\in (2,4),\\
\eps\sqrt{|\log \eps|} & \beta=4,\\
\eps & \beta\in (4,d).
\end{array} \right.
\end{equation}

%
%

  \end{itemize}
  \label{thm:THgaussian}
\end{theorem}

The result shows that for sufficiently long-range-correlated random potentials, the convergence rate in homogenization could be potential-dependent, e.g., when $\beta\to 2$, the error is of order $\eps^{\frac{\beta}{2}-1}$ and could be arbitrarily close to $O(1)$. On the other hand, it can be shown that when the covariance function is integrable, i.e., $\beta>d$, we recover the result for strongly mixing potentials.

\subsection{Convergence to SPDE}

Let $(\Omega,\mathcal{F},\Pb)$ be a probability space. The following is our assumption on random coefficient $V(x)=V(x,\omega)$ with $\omega\in\Omega$ labeling the particular realization.
\begin{assumption}
$V(x)=\Phi(g(x))$, where
\begin{itemize}
\item $g(x)$ is a stationary Gaussian field with zero mean and unit variance. The auto-covariance function $R_g(x)=\E\{g(0)g(x)\}$ satisfies that $|R_g(x)|\les \prod_{i=1}^d \min(1,|x_i|^{-\alpha_i})$ with $\alpha_i\in (0,1)$ and $R_g(x)\sim c_d\prod_{i=1}^d |x_i|^{-\alpha_i}$ as $\min_{i=1,\ldots,d}|x_i|\to \infty$. $\alpha:=\sum_{i=1}^d\alpha_i\in (0,2)$.

\item $\Phi$ is a deterministic function with Hermite rank $1$, i.e., $\int_{\R}\Phi^2(x)\frac{1}{\sqrt{2\pi}}\exp(-x^2/2)dx<\infty$ and if we define $V_k=\E\{\Phi(g)H_k(g)\}$ with $H_k(x)=(-1)^k\exp(x^2/2)\frac{d^k}{dx^k}\exp(-x^2/2)$ the $k-$th Hermite polynomial, then $V_0=0,V_1\neq 0$.

\end{itemize}
\label{ass:longRangeGauss}
\end{assumption}

We will see later that $R(x)=\E\{V(0)V(x)\}\sim V_1^2c_d \prod_{i=1}^d |x_i|^{-\alpha_i}$, and since $\alpha=\sum_{i=1}^d \alpha_i<2$, $R(x)|x|^{2-d}$ is not integrable, so $\sigma^2$ in \eqref{eq:limitHOMO} is not well-defined and we do not expect the result of homogenization. We consider the equation when $d\geq 3$:
\begin{equation}
\partial_t u_\eps(t,x,\omega)=\frac12\Delta u_\eps(t,x,\omega)+i\frac{1}{\eps^{\alpha/2}}V(\frac{x}{\eps},\omega)u_\eps(t,x,\omega),
\label{eq:mainEqSPDE}
\end{equation}
with initial condition $u_\eps(0,x,\omega)=f(x)$ for $f\in \C_b(\R^d)$, i.e., in \eqref{eq:mainEq}, we choose $\gamma=\frac{\alpha}{2}<1$. The following is the result of convergence to SPDE.

\begin{theorem}
Under Assumption \ref{ass:longRangeGauss}, we have $u_\eps(t,x)\to u_{spde}(t,x)$ in distribution, with $u_{spde}$ solving the SPDE with multiplicative noise:
\begin{equation}
\partial_t u_{spde}=\frac12\Delta u_{spde}+i V_1\sqrt{c_d}\dot{W} u_{spde},
\label{eq:limitSPDE}
\end{equation}
where $\dot{W}(x)$ is a generalized Gaussian random field with covariance function $\E\{\dot{W}(x)\dot{W}(y)\}=\prod_{i=1}^d |x_i-y_i|^{-\alpha_i}$.
\label{thm:THspde}
\end{theorem}

For the limiting SPDE \eqref{eq:limitSPDE}, the product between $\dot{W}$ and $u_{spde}$ is in the Stratonovich's sense. The solution will be defined through a Feynman-Kac formula and shown to be a weak solution.

\begin{remark}
The proof of Theorem \ref{thm:THspde} also holds for $d=1,2$. When $d=2$, since $\alpha_1,\alpha_2\in (0,1)$, $\alpha=\alpha_1+\alpha_2\in (0,2)$ is automatically satisfied. When $d=1$, we have $\alpha=\alpha_1\in (0,1)$.
\end{remark}

\subsection{Remarks}
\label{sec:conclu}

One of the  main ingredients in the proof of both homogenization and convergence to SPDE is the weak convergence of Brownian motion in random scenery. In the homogenization setting, Kipnis-Varadhan's result implies $\eps^{-1}\int_0^t V(B_s/\eps)ds\Rightarrow \sigma W_t$ in $\C([0,T])$ in $\Pb-$probability, with only necessary assumptions of stationarity, ergodicity, and finiteness of asymptotic variance. In the SPDE setting, Proposition \ref{prop:conLong} below shows $\eps^{-\alpha/2}\int_0^t V(B_s/\eps)ds\Rightarrow V_1\sqrt{c_d}\int_0^t \dot{W}(B_s)ds$ in the annealed sense, where $V$ is chosen as functionals of stationary Gaussian process. The difference between the results of weak convergence sheds light on the transition from homogenization to stochasticity from a probabilistic point of view. 

To obtain optimal error estimate, a quantification of ergodicity is in need and we assume a strong mixing of the random potential. Ergodicity is quantified by controlling the tail of the mixing coefficient and is used only to estimate the fourth-order moment of the random potential. When the fourth-order moment can be estimated explicitly without any mixing condition, then similar error estimates can be derived. We considered here the example of long-range-correlated Gaussian potential and derived convergence rate depending on its decorrelation rate.

In the homogenization setting of low dimensions, when $d=2$, weak convergence of Brownian motion in random scenery has been proved for specific types of short-range-correlated potentials in the annealed sense, including Gaussian, Poissonian \cite{gu2014invariance} and piecewise-constant cases \cite{remillard1991limit}. The size of potentials then includes a logarithm correction. It is not clear whether Kipnis-Varadhan's approach works to obtain weak convergence in probability. With the annealed weak convergence, homogenization could be derived by showing the convergence of $\E\{u_\eps(t,x)\}$ and $\E\{|u_\eps(t,x)|^2\}$ respectively. When $d=1$, \cite{pardoux2006homogenization} derived a stochastic limit for short-range-correlated potentials.

Intuitively, Theorem \ref{thm:THhomo} of homogenization corresponds to law of large numbers while Theorem \ref{thm:THcorrector} and \ref{thm:THgaussian} relate to error estimate. It is natural to inquire about central limit type result, i.e., the weak convergence of $\eps^{-\delta}(u_\eps(t,x)-u_{hom}(t,x))$ for appropriate $\delta>0$. In \cite{bal2010homogenization}, for the same type of equations, central limit type of result is derived by a different approach for Gaussian potentials. The probabilistic approach is currently under study.


\section{Proof of homogenization and error estimate}
\label{sec:homo}

\subsection{Feynman-Kac formula, medium seen from the observer and auxiliary equation}
\label{sec:fkSetup}
The solution to \eqref{eq:mainEqhomo} is written as
\begin{equation}
u_\eps(t,x)=\E_B\{f(x+B_t)\exp(i\frac{1}{\eps}\int_0^tV(\frac{x+B_s}{\eps})ds)\},
\label{eq:soFK}
\end{equation}
with Brownian motion $B_s$ starting from the origin. 


By the scaling property of Brownian motion, $$u_\eps(t,x)=\E_B\{f(x+\eps B_{t/\eps^2})\exp(i\eps\int_0^{t/\eps^2}V(\frac{x}{\eps}+B_s)ds)\}.$$ Since $u_{hom}$ is deterministic, by stationarity of $V$, the difference between the solutions to the heterogeneous and homogenized equations can be written as
\begin{equation}
\begin{aligned}
&\E\{|u_\eps(t,x)-u_{hom}(t,x)|\}\\
=&\E\{|\E_B\{f(x+\eps B_{t/\eps^2})\exp(i\eps\int_0^{t/\eps^2}V(B_s)ds)\}-\E_B\{f(x+\eps B_{t/\eps^2})\exp(-\frac12\sigma^2t)\}|\}.
\end{aligned}
\end{equation}
Now we look at $X_\eps(t):=\eps\int_0^{t/\eps^2}\V(\tau_{B_s}\omega)ds=\eps\int_0^{t/\eps^2}V(B_s)ds$. For $y_s:=\tau_{B_s}\omega$, it is a stationary, ergodic Markov process taking values in $\Omega$ with invariant measure $\Pb$, and the generator of $y_s$ is given by $L=\frac12\sum_{k=1}^dD_k^2$, see e.g. \cite{komorowski2012fluctuations}.

We define the corrector function $\Phi_\lambda$ for any $\lambda>0$ such that 
\begin{equation}
(\lambda I-L)\Phi_\lambda=\V,
\label{eq:correctorEq}
\end{equation}
then the following proposition holds.

\begin{proposition}
\begin{equation}
\Phi_\lambda=\int_{\R^d}\frac{1}{\lambda+\frac12|\xi|^2}U(d\xi)\V.
\end{equation}

Under Assumption \ref{ass:finiteVar}, $\lambda\langle\Phi_\lambda,\Phi_\lambda\rangle\to 0$ as $\lambda\to 0$.

Under Assumption \ref{ass:mixing},
\begin{equation}
\lambda\langle \Phi_\lambda,\Phi_\lambda\rangle\les\left\{
\begin{array}{ll}
\sqrt{\lambda} & d=3,\\
\lambda|\log \lambda| & d=4,\\
\lambda & d>4.
\end{array}\right.
\end{equation}

Under Assumption \ref{ass:gauAss},
\begin{equation}
\lambda\langle \Phi_\lambda,\Phi_\lambda\rangle\les\left\{
\begin{array}{ll}
\lambda^{\frac{\beta}{2}-1} & \beta\in (2,4),\\
\lambda|\log \lambda| & \beta=4,\\
\lambda & \beta>4.
\end{array}\right.
\end{equation}
\label{prop:remainder1}
\end{proposition}

If we define
\begin{equation}
\eta_k=\int_{\R^d}\frac{2i\xi_k}{|\xi|^2}U(d\xi)\V
\end{equation}
for $k=1,\ldots,d$, $\sigma^2=\sum_{k=1}^d \|\eta_k\|^2$. Defining $\sigma_\lambda^2=\sum_{k=1}^d \|D_k\Phi_\lambda\|^2$, the following proposition holds.
\begin{proposition}
Under Assumption \ref{ass:finiteVar}, $D_k\Phi_\lambda\to \eta_k$ in $L^2(\Omega)$ as $\lambda\to 0$.

Under Assumption \ref{ass:mixing},
\begin{equation}
|\sigma_\lambda^2-\sigma^2|\les \left\{
\begin{array}{ll}
\sqrt{\lambda} & d=3,\\
\lambda |\log \lambda| & d=4,\\
\lambda & d>4.
\end{array}\right.
\end{equation}

Under Assumption \ref{ass:gauAss},
\begin{equation}
|\sigma_\lambda^2-\sigma^2|\les \left\{
\begin{array}{ll}
\lambda^{\frac{\beta}{2}-1} & \beta\in (2,4),\\
\lambda |\log \lambda| & \beta=4.\\
\lambda & \beta>4.
\end{array}\right.
\end{equation}

\label{prop:remainder2}
\end{proposition}

\begin{proof}[Proof of Proposition \ref{prop:remainder1}]

First, we have
\begin{equation}
\lambda\langle \Phi_\lambda,\Phi_\lambda\rangle=\int_{\R^d}\frac{\lambda}{\lambda+\frac12|\xi|^2}\frac{\hat{R}(\xi)}{\lambda+\frac12|\xi|^2}d\xi\les
\int_{\R^d} \frac{\lambda}{\lambda+|\xi|^2}\frac{\hat{R}(\xi)}{|\xi|^2}d\xi.
\end{equation}
Under Assumption \ref{ass:finiteVar}, i.e., $\hat{R}(\xi)|\xi|^{-2}$ is integrable, by the dominated convergence theorem, $\lambda\langle \Phi_\lambda,\Phi_\lambda\rangle\to 0$ as $\lambda \to 0$.

If Assumption \ref{ass:mixing} holds, $\hat{R}(\xi)$ is bounded, and we obtain by direct calculation:
\begin{equation}
\begin{aligned}
\lambda\langle \Phi_\lambda,\Phi_\lambda\rangle\les&\lambda^{\frac{d}{2}-1}\int_{\R^d}\frac{1}{1+|\xi|^2}\frac{\hat{R}(\sqrt{\lambda}\xi)}{|\xi|^2}d\xi\\
\les & \lambda^{\frac{d}{2}-1}\int_{\sqrt{\lambda}|\xi|<1}\frac{1}{1+|\xi|^2}\frac{1}{|\xi|^2}d\xi+\lambda\int_{|\xi|>1}\frac{\hat{R}(\xi)}{|\xi|^4}d\xi\\
\les &\lambda^{\frac{d}{2}-1}\int_0^{\frac{1}{\sqrt{\lambda}}}\frac{r^{d-3}}{1+r^2}dr+\lambda,
\end{aligned}
\end{equation}
so when $d=3$, $\lambda\langle \Phi_\lambda,\Phi_\lambda\rangle\les \sqrt{\lambda}$. When $d=4$, $\lambda\langle \Phi_\lambda,\Phi_\lambda\rangle\les \lambda|\log \lambda|$. When $d>4$, $\lambda\langle \Phi_\lambda,\Phi_\lambda\rangle\les \lambda$.

If Assumption \ref{ass:gauAss} holds, $\hat{R}(\xi)\les |\xi|^{\beta-d}$ at the origin, and the proof is similar.
\end{proof}

\begin{proof}[Proof of Proposition \ref{prop:remainder2}]
Since
\begin{equation}
\|D_k\Phi_\lambda-\eta_k\|^2=\frac{1}{(2\pi)^d}\int_{\R^d}\frac{\lambda^2\xi_k^2}{(\lambda+\frac12|\xi|^2)^2\frac14|\xi|^4}\hat{R}(\xi)d\xi\les \int_{\R^d}\frac{\lambda^2}{\lambda^2+|\xi|^4}\frac{\hat{R}(\xi)}{|\xi|^2}d\xi,
\end{equation}
and
\begin{equation}
\sigma_\lambda^2-\sigma^2=-\frac{16}{(2\pi)^d}\int_{\R^d}\frac{(\lambda^2+\lambda|\xi|^2)}{|\xi|^2(2\lambda+|\xi|^2)^2}\hat{R}(\xi)d\xi,
\end{equation}
we obtain the result as in the proof of Proposition \ref{prop:remainder1}.
\end{proof}


Now we are ready to prove the main theorems. We choose $\lambda=\eps^2$ from now on.

By It\^o's formula, the process of Brownian motion in random scenery can be decomposed as \begin{equation}
X_\eps(t)=\eps\int_0^{t/\eps^2}\V(\tau_{B_s}\omega)ds=R^\eps_t+M^\eps_t,
\end{equation} where
\begin{eqnarray}
R^\eps_t:&=&\eps\int_0^{t/\eps^2}\lambda\Phi_\lambda(y_s)ds-\eps\Phi_\lambda(y_{t/\eps^2})+\eps\Phi_\lambda(y_0),\\
M^\eps_t:&=&\eps\int_0^{t/\eps^2}\sum_{k=1}^d D_k\Phi_\lambda(y_s)dB^k_s.
\end{eqnarray}
Therefore, the error is decomposed correspondingly as
$
u_\eps(t,x)-u_{hom}(t,x)=\mathcal{E}_1+\mathcal{E}_2,
$
where
\begin{eqnarray}
\mathcal{E}_1&=& \E_B \{f(x+\eps B_{t/\eps^2})\exp(iR^\eps_t+iM^\eps_t)\}-\E_B\{f(x+\eps B_{t/\eps^2})\exp(iM^\eps_t)\},\\
\mathcal{E}_2&=& \E_B\{f(x+\eps B_{t/\eps^2})\exp(iM^\eps_t)\}-\E_B\{f(x+\eps B_{t/\eps^2})\exp(-\frac12\sigma^2t)\}.
\end{eqnarray}
We see $\mathcal{E}_1$ is caused by the residue $R^\eps_t$, i.e., a measure of how close $X_\eps(t)$ is to a martingale, while $\mathcal{E}_2$ relates to the convergence of the martingale $M^\eps_t$, i.e., a measure of how close the martingale is to a Brownian motion. Since $f$ is bounded, we have the estimate $\E\{|\mathcal{E}_1|\}\les\E\E_{B}\{|R^\eps_t|\}$. It is straightforward to check that
\begin{equation}
\E\{|\mathcal{E}_1|\}\les\E\E_{B}\{|R^\eps_t|\}\les\sqrt{ \lambda\langle\Phi_\lambda,\Phi_\lambda\rangle}(1+t).
\label{eq:EsReps}
\end{equation}

In the following, we estimate the convergence of $M^\eps_t$ to a Brownian motion in different ways to prove homogenization and error estimate respectively.

\subsection{Homogenization: proof of Theorem \ref{thm:THhomo}}

We rewrite $$M^\eps_t=\eps\int_0^{t/\eps^2}\sum_{k=1}^d (D_k\Phi_\lambda-\eta_k)(y_s)dB^k_s+\eps\int_0^{t/\eps^2}\sum_{k=1}^d \eta_k(y_s)dB^k_s:=\mathcal{E}_3+\mathcal{E}_4,$$
so
\begin{equation}
\begin{aligned}
|\mathcal{E}_2|\leq&|\E_B\{f(x+\eps B_{t/\eps^2})\exp(iM^\eps_t)\}-\E_B\{f(x+\eps B_{t/\eps^2})\exp(i\mathcal{E}_4)\}|\\
+&|\E_B\{f(x+\eps B_{t/\eps^2})\exp(i\mathcal{E}_4)\}-\E_B\{f(x+\eps B_{t/\eps^2})\exp(-\frac12\sigma^2t)\}|\\
\les &\E_{B}\{|\mathcal{E}_3|\}+\E_B\{f(x+\eps B_{t/\eps^2})\exp(i\mathcal{E}_4)\}-\E_B\{f(x+\eps B_{t/\eps^2})\exp(-\frac12\sigma^2t)\}.
\end{aligned}
\end{equation}

On one hand, we clearly have that
\begin{equation}
\E\E_{B}\{|\mathcal{E}_3|\}\leq \sqrt{t\sum_{k=1}^d \|D_k\Phi_\lambda-\eta_k\|^2}.
\end{equation}

On the other hand, by ergodic theorem and the fact that $\E\{\eta_k\}=0$ for $k=1,\ldots,d$, and $\sum_{k=1}^d \|\eta_k\|^2=\sigma^2$, we obtain for almost every $\omega\in \Omega$ and $k=1,\ldots,d$: 
\begin{eqnarray*}
\eps^2\int_0^{t/\eps^2}\eta_k(\tau_{B_s}\omega)ds&\to& 0\\
\eps^2\int_0^{t/\eps^2}\sum_{k=1}^d\eta_k^2(\tau_{B_s}\omega)ds&\to& \sigma^2
\end{eqnarray*} almost surely. Now by martingale central limit theorem \cite[page 339, Theorem 1.4]{ethier2009markov}, we conclude for almost every $\omega\in\Omega$ that:
\begin{equation}
(\eps B_{t/\eps^2}, \eps\int_0^{t/\eps^2}\sum_{k=1}^d \eta_k(\tau_{B_s}\omega)dB^k_s)\Rightarrow (W^1_t, \sigma W^2_t),
\end{equation}
where $W^1_t$ is a $d-$dimensional Brownian motion and $W^2_t$ is an independent $1-$dimensional Brownian motion. Therefore,
\begin{equation}
\E_B\{f(x+\eps B_{t/\eps^2})\exp(i\mathcal{E}_4)\}-\E_B\{f(x+\eps B_{t/\eps^2})\exp(-\frac12\sigma^2t)\}\to 0
\end{equation}
as $\eps \to 0$ almost surely.

To summarize, we have
\begin{equation}
\begin{aligned}
\E\{|u_\eps(t,x)-u_{hom}(t,x)|\}\les &\sqrt{\lambda\langle \Phi_\lambda,\Phi_\lambda\rangle}(1+t)
+\sqrt{t\sum_{k=1}^d \|D_k\Phi_\lambda-\eta_k\|^2}\\
+&\E\{|\E_B\{f(x+\eps B_{t/\eps^2})\exp(i\mathcal{E}_4)\}-\E_B\{f(x+\eps B_{t/\eps^2})\exp(-\frac12\sigma^2t)\}|\}.
\end{aligned}
\end{equation}
By Proposition \ref{prop:remainder1} and \ref{prop:remainder2}, and the dominated convergence theorem, the proof of Theorem \ref{thm:THhomo} is complete.


\subsection{Error estimate: proof of Theorem \ref{thm:THcorrector} and \ref{thm:THgaussian}}

Defining $\hat{f}(\xi)=\int_{\R^d}f(x)e^{-i\xi\cdot x}dx$, we can write $\mathcal{E}_2$ as
\begin{equation}
\begin{aligned}
\mathcal{E}_2
=\frac{1}{(2\pi)^d}\int_{\R^d}\hat{f}(\xi)e^{i\xi\cdot x}\E_B\{e^{i(\eps \xi\cdot B_{t/\eps^2}+M^\eps_t)}-e^{i\eps \xi\cdot B_{t/\eps^2}-\frac12\sigma^2t}\}d\xi,
\end{aligned}
\end{equation}
where $\eps\xi\cdot B_{t/\eps^2}+M^\eps_t=\eps\int_0^{t/\eps^2}\sum_{k=1}^d (\xi_k+D_k\Phi_\lambda(y_s))dB^k_s$ is a continuous, square-integrable martingale for almost every $\omega\in \Omega$. The estimation of $\E_B\{e^{i(\eps \xi\cdot B_{t/\eps^2}+M^\eps_t)}-e^{i\eps \xi\cdot B_{t/\eps^2}-\frac12\sigma^2t}\}$ reduces to a control of the Wasserstein distance between $\eps \xi\cdot B_{t/\eps^2}+M^\eps_t$ and $\eps \xi\cdot B_{t/\eps^2}+\sigma W_t$, where $W_t$ is an independent Brownian motion from $B_t$. A general quantitative martingale central limit theorem is proved in \cite{mourrat2012kantorovich}, from which we extract the following result for continuous martingales.

\begin{proposition}[Theorem 3.2, \cite{mourrat2012kantorovich}]
If $M_t$ is a continuous martingale and $W_t$ is a standard Brownian motion, then
\begin{equation}
d_{1,k}(M_1,W_1)\leq (1\vee k)\E\{ |\langle M\rangle_1-1|\},
\end{equation}
with the distance $d_{1,k}$ defined as
\begin{equation}
d_{1,k}(X,Y)=\sup \{|\E\{f(X)-f(Y)\}|: f\in C_b^2(\R), \|f'\|_\infty\leq 1, \|f''\|_\infty\leq k\}.
\end{equation}
\label{prop:qmCLT}
\end{proposition}

For the sake of convenience, we present the proof in the Appendix.

Since $\sigma_\lambda^2=\sum_{k=1}^d \langle D_k\Phi_\lambda,D_k\Phi_\lambda\rangle$, by Proposition \ref{prop:qmCLT} we have for almost every $\omega\in\Omega$:
\begin{equation}
\begin{aligned}
&|\E_B\{e^{i(\eps \xi\cdot B_{t/\eps^2}+M^\eps_t)}\}-e^{-\frac12(|\xi|^2+\sigma_\lambda^2)t}|\\
\leq &\left(1\vee \frac{1}{\sqrt{(|\xi|^2+\sigma_\lambda^2)t}}\right)\E_B\{|\eps^2\int_0^{t/\eps^2}\sum_{k=1}^d(\xi_k+D_k\Phi_\lambda(y_s))^2ds-(|\xi|^2+\sigma_\lambda^2)t|\}.
\end{aligned}
\end{equation}
Now we can write
$
|\mathcal{E}_2|\leq \mathcal{E}_5+\mathcal{E}_6,
$
where
\begin{eqnarray*}
\mathcal{E}_5&=& \frac{1}{(2\pi)^d}\int_{\R^d}|\hat{f}(\xi)|\left(1\vee \frac{1}{\sqrt{(|\xi|^2+\sigma_\lambda^2)t}}\right)\E_B\{|\eps^2\int_0^{t/\eps^2}\sum_{k=1}^d(\xi_k+D_k\Phi_\lambda(y_s))^2ds-(|\xi|^2+\sigma_\lambda^2)t|\}d\xi,\\
\mathcal{E}_6&=&\frac{1}{(2\pi)^d}\int_{\R^d}|\hat{f}(\xi)| |e^{-\frac12(|\xi|^2+\sigma_\lambda^2)t}-e^{-\frac12(|\xi|^2+\sigma^2)t}|d\xi.
\end{eqnarray*}

First, we have \begin{equation}
\mathcal{E}_6\les |\sigma_\lambda^2-\sigma^2|t.
\label{eq:deterErr}
\end{equation}

Secondly, we rewrite
\begin{equation}
\begin{aligned}
\mathcal{E}_5
 = \frac{1}{(2\pi)^d}\int_{\R^d}|\hat{f}(\xi)|\left(1\vee \frac{1}{\sqrt{(|\xi|^2+\sigma_\lambda^2)t}}\right)\E_B\{|\eps^2\int_0^{t/\eps^2}Z_{\lambda,\xi}(B_s)ds|\}d\xi,
\end{aligned}
\end{equation}
where $$Z_{\lambda,\xi}(x):=\sum_{k=1}^d(\xi_k+\int_{\R^d}\partial_{x_k}G_\lambda(x-y)V(y)dy)^2-|\xi|^2-\sigma_\lambda^2,$$ with $G_\lambda$ the Green's function of $\lambda-\frac12\Delta$. Note that we have used the fact that $$D_k\Phi_\lambda(\tau_x\omega)=\int_{\R^d}\partial_{x_k}G_\lambda(x-y)V(y)dy.$$ Clearly $Z_{\lambda,\xi}$ has zero mean; and by the ergodic theorem, we expect $\eps^2\int_0^{t/\eps^2}Z_{\lambda,\xi}(B_s)ds$ to be small. This is quantified by the following control of the variance of Brownian motion in random scenery.

\begin{lemma}
If $V$ is a mean zero, stationary random field with covariance function $R(x)$, and $B_s$ is Brownian motion independent from $V$, then
\begin{equation}
\E\E_{B}\{\left(\eps\int_0^{t/\eps^2}V(B_s)ds\right)^2\}\les t\int_{\R^d}\frac{|R(x)|}{|x|^{d-2}}dx.
\end{equation}
\label{lem:VarBMRS}
\end{lemma}

\begin{proof}
By direct calculation, we have
\begin{equation}
\begin{aligned}
\E\E_{B}\{\left(\eps\int_0^{t/\eps^2}V(B_s)ds\right)^2\}
=&2\eps^2\int_0^{t/\eps^2}\int_0^s\int_{\R^d}R(x)\frac{1}{(2\pi u)^{\frac{d}{2}}}e^{-\frac{|x|^2}{2u}}dxduds\\
=&2\eps^2\int_0^\infty du(\frac{t}{\eps^2}-u)1_{u<\frac{t}{\eps^2}}\int_{\R^d}R(x)\frac{1}{(2\pi u)^{\frac{d}{2}}}e^{-\frac{|x|^2}{2u}}dx\\
=&\eps^2\int_0^\infty d\lambda(\frac{t}{\eps^2}-\frac{|x|^2}{2\lambda})1_{\frac{|x|^2}{2\lambda}<\frac{t}{\eps^2}}\lambda^{\frac{d}{2}-2}e^{-\lambda}\int_{\R^d}\frac{1}{\pi^{\frac{d}{2}}}R(x)\frac{1}{|x|^{d-2}}dx\\
\les& t\int_{\R^d}\frac{|R(x)|}{|x|^{d-2}}dx.
\end{aligned}
\end{equation}
\end{proof}

Now we write $Z_{\lambda,\xi}(x)=Z_{1,\lambda,\xi}(x)+Z_{2,\lambda,\xi}(x)$ with
\begin{eqnarray}
Z_{1,\lambda,\xi}(x)&=&\sum_{k=1}^d\left(\int_{\R^d}\partial_{x_k}G_\lambda(x-y)V(y)dy\right)^2-\sigma_\lambda^2,\\
Z_{2,\lambda,\xi}(x)&=&2\sum_{k=1}^d \xi_k\int_{\R^d}\partial_{x_k}G_\lambda(x-y)V(y)dy.
\end{eqnarray}
Since $\sigma_\lambda^2=\sum_{k=1}^d\langle D_k\Phi_\lambda,D_k\Phi_\lambda\rangle$, we have $\E\{Z_{i,\lambda,\xi}(x)\}=0, i=1,2$. Therefore, Lemma \ref{lem:VarBMRS} implies\begin{equation}
\E\{\mathcal{E}_5\}\les \frac{\eps}{(2\pi)^d}\int_{\R^d}|\hat{f}(\xi)|\left(1\vee \frac{1}{\sqrt{(|\xi|^2+\sigma_\lambda^2)t}}\right)\sqrt{t\int_{\R^d}\frac{|\RR_{1,\lambda,\xi}(x)|+|\RR_{2,\lambda,\xi}(x)|}{|x|^{d-2}}dx}d\xi
\end{equation}
where $\RR_{i,\lambda,\xi}(x):=\E\{Z_{i,\lambda,\xi}(0)Z_{i,\lambda,\xi}(x)\}, i=1,2$. By recalling \eqref{eq:EsReps} and \eqref{eq:deterErr}, we have\begin{equation}
\begin{aligned}
&\E\{|u_\eps(t,x)-u_{hom}(t,x)|\}\\
\les &\left(\sqrt{\lambda\langle \Phi_\lambda,\Phi_\lambda\rangle}+|\sigma_\lambda^2-\sigma^2|+\eps\int_{\R^d}\hat{f}(\xi)\sqrt{\int_{\R^d}\frac{|\RR_{1,\lambda,\xi}(x)|+|\RR_{2,\lambda,\xi}(x)|}{|x|^{d-2}}dx}d\xi\right)(1+t).
\end{aligned}
\label{eq:Es3terms}
\end{equation}

The estimation of $\RR_{i,\lambda,\xi}$ is done for strongly mixing potentials and long-range-correlated Gaussian potentials respectively in the following sections.

\subsubsection{Strongly mixing case: proof of Theorem \ref{thm:THcorrector}}

Defining 
\begin{equation}
F_{\lambda,c,\beta}(x):=\lambda^{\frac{d}{2}-1}e^{-c\sqrt{\lambda}|x|}+1\wedge \frac{e^{-c\sqrt{\lambda}|x|}}{|x|^{d-2}}+1\wedge\frac{1}{|x|^\beta}
\end{equation}
for $c,\beta>0$, we have the following result.
\begin{proposition}
Under Assumption \ref{ass:mixing}, there exist a constant $c>0$ and a sufficiently large $\beta>0$ such that
\begin{equation}
|\RR_{1,\lambda,\xi}(x)|+|\RR_{2,\lambda,\xi}(x)|\les
(1+|\xi|)^2F_{\lambda,c,\beta}(x).
\end{equation}
\label{prop:boundCOV}
\end{proposition}

\begin{proof}
We first consider $\RR_{1,\lambda,\xi}(x)$. By denoting $\phi_\lambda(x)=\int_{\R^d}G_\lambda(x-y)V(y)dy$, for any $m,n=1,\ldots,d$, we have
\begin{equation}
\begin{aligned}
&\E\{(\partial_{x_m}\phi_\lambda(0))^2(\partial_{x_n}\phi_\lambda(x))^2\}\\
=&\int_{\R^{4d}}\partial_{x_m}G_\lambda(y_1)\partial_{x_m}G_\lambda(z_1)\partial_{x_n}G_\lambda(y_2)\partial_{x_n}G_\lambda(z_2)\E\{V(-y_1)V(-z_1)V(x-y_2)V(x-z_2)\}dy_1dy_2dz_1dz_2\\
=&\int_{\R^{4d}}\partial_{x_m}G_\lambda(y_1)\partial_{x_m}G_\lambda(z_1)\partial_{x_n}G_\lambda(y_2)\partial_{x_n}G_\lambda(z_2)R(y_1-z_1)R(y_2-z_2)dy_1dy_2dz_1dz_2+I_{mn}\\
=&\|D_m\Phi_\lambda\|^2\|D_n\Phi_\lambda\|^2+I_{mn},
\end{aligned}
\end{equation}
where $I_{mn}$ are remainders in the calculation of fourth moment. By Lemma \ref{lem:34Moment}, we obtain
\begin{equation}
|I_{mn}|\leq 2\int_{\R^{4d}}|\partial_mG_\lambda(y_1)\partial_mG_\lambda(z_1)\partial_nG_\lambda(y_2)\partial_nG_\lambda(z_2)|\Psi(x-y_1+y_2)\Psi(x-z_1+z_2)dy_1dy_2dz_1dz_2,
\end{equation}
where $|\Psi(x)|\les 1\wedge |x|^{-\beta}$ for any $\beta>0$. Since $G_\lambda$ is the Green's function of $\lambda-\frac12\Delta$, by scaling property, $G_\lambda(x)=\lambda^{\frac{d}{2}-1}G_1(\sqrt{\lambda}x)$. The estimate $|\nabla G_1(x)|\les e^{-\rho|x|}|x|^{1-d}$ holds for some $\rho>0$  \cite[page 271, (6.49)]{taylor2011partial}. Therefore, by change of variables, we have
\begin{equation}
|I_{mn}|\les \left(\frac{1}{\lambda}\int_{\R^{2d}}\frac{e^{-\rho|y|}}{|y|^{d-1}}\frac{e^{-\rho|z|}}{|z|^{d-1}}\Psi(x-\frac{y-z}{\sqrt{\lambda}})dydz\right)^2.
\end{equation}
Since $\sigma_\lambda^4=\sum_{m,n=1}^d \|D_m\Phi_\lambda\|^2\|D_n\Phi_\lambda\|^2$, we derive the following estimate
\begin{equation}
|\RR_{1,\lambda,\xi}(x)|\les  \left(\frac{1}{\lambda}\int_{\R^{2d}}\frac{e^{-\rho|y|}}{|y|^{d-1}}\frac{e^{-\rho|z|}}{|z|^{d-1}}\Psi(x-\frac{y-z}{\sqrt{\lambda}})dydz\right)^2.
\end{equation}

Now we consider $\RR_{2,\lambda,\xi}(x)$. Similary, we obtain that
\begin{equation}
\begin{aligned}
|\RR_{2,\lambda,\xi}(x)|=&|4\sum_{m,n=1}^d \xi_m\xi_n\int_{\R^{2d}}\partial_mG_\lambda(y)\partial_nG_\lambda(z)R(x-y+z)dydz|\\
\les& |\xi|^2\frac{1}{\lambda}\int_{\R^{2d}}\frac{e^{-\rho|y|}}{|y|^{d-1}}\frac{e^{-\rho|z|}}{|z|^{d-1}}|R|(x-\frac{y-z}{\sqrt{\lambda}})dydz.
\end{aligned}
\end{equation}

Since $|\Psi(x)|\les 1\wedge |x|^{-\beta}$ for $\beta>0$ sufficiently large, by Lemma \ref{lem:2ndMoment}, we obtain
\begin{equation}
|\RR_{1,\lambda,\xi}(x)|+|\RR_{2,\lambda,\xi}(x)|\les(1+|\xi|)^2F_{\lambda,c,\beta}(x)
\end{equation}
for some constant $c>0$, and $\beta>0$ sufficiently large. The proof is complete.
\end{proof}

By combining Proposition \ref{prop:boundCOV} and \eqref{eq:Es3terms}, we obtain that
\begin{equation}
\begin{aligned}
&\E\{|u_\eps(t,x)-u_{hom}(t,x)|\}\\
\les &\left(\sqrt{\lambda\langle \Phi_\lambda,\Phi_\lambda\rangle}+|\sigma_\lambda^2-\sigma^2|+\eps\sqrt{\int_{\R^d}\frac{F_{\lambda,c,\beta}(x)}{|x|^{d-2}}dx}\right)(1+t).
\end{aligned}
\label{eq:Es3terms1}
\end{equation}
We also see that for the initial condition $f$, the only requirement is $|\hat{f}(\xi)| (1+|\xi|)$ being integrable.

By Proposition \ref{prop:remainder1} and \ref{prop:remainder2} and $\lambda=\eps^2$, we have under Assumption \ref{ass:mixing}
\begin{equation}
\sqrt{\lambda\langle\Phi_\lambda,\Phi_\lambda\rangle}+|\sigma_\lambda^2-\sigma^2|\les
\left\{
\begin{array}{ll}
\sqrt{\eps} & \mbox{ $d=3$},\\
\eps\sqrt{|\log \eps|} & \mbox{ $d=4$},\\
\eps & \mbox{ $d>4$}.
\end{array}\right.
\label{eq:remainderEstimate}
\end{equation}
Together with the following Lemma \ref{lem:fluctuationError} and \eqref{eq:Es3terms1}, the proof of Theorem \ref{thm:THcorrector} is complete.

\begin{lemma}
\begin{equation}
\int_{\R^d}\frac{F_{\lambda,c,\beta}(x)}{|x|^{d-2}}dx\les \left\{
\begin{array}{ll}
\lambda^{-\frac12} & d=3,\\
|\log\lambda| & d=4,\\
1 & d>4.
\end{array}\right.
\end{equation}
\label{lem:fluctuationError}
\end{lemma}
\begin{proof}
Note that $1\wedge |x|^{-\beta}$ gives a term of order $1$ since $\beta$ could be sufficiently large. We first look at
\begin{equation}
\int_{\R^d}\frac{1}{|x|^{d-2}}\lambda^{\frac{d}{2}-1}e^{-c\sqrt{\lambda}|x|}dx=\lambda^{\frac{d}{2}-2}\int_{\R^d}\frac{e^{-c|y|}}{|y|^{d-2}}dy\les \lambda^{\frac{d}{2}-2}.
\end{equation}
Now we only have to deal with $1\wedge \frac{e^{-c\sqrt{\lambda}|x|}}{|x|^{d-2}}$.
\begin{equation}
\int_{\R^d}\frac{1}{|x|^{d-2}}1\wedge \frac{e^{-c\sqrt{\lambda}|x|}}{|x|^{d-2}} dx
\leq \int_{|x|<1}\frac{1}{|x|^{d-2}}dx+\int_{|x|>1}\frac{e^{-c\sqrt{\lambda}|x|}}{|x|^{2d-4}}dx.
\end{equation}
When $d>4$, RHS is bounded. When $d\leq 4$,
\begin{equation}
\int_{|x|>1}\frac{e^{-c\sqrt{\lambda}|x|}}{|x|^{2d-4}}dx=\lambda^{\frac{d-4}{2}}\int_{\sqrt{\lambda}}^\infty \frac{e^{-cr}}{r^{d-3}}dr,
\end{equation}
which concludes the proof.
\end{proof}

\subsubsection{Long-range-correlated Gaussian case: proof of Theorem \ref{thm:THgaussian}}

If we follow the proof of Proposition \ref{prop:boundCOV}, it is straightforward to check that when $V$ is Gaussian, the following estimate holds:
\begin{equation}
|\RR_{1,\lambda,\xi}(x)|+|\RR_{2,\lambda,\xi}(x)|\les (1+|\xi|)^2 (F_{\lambda,\rho}(x)+F_{\lambda,\rho}^2(x)),
\end{equation}
with $$F_{\lambda,\rho}(x):=\frac{1}{\lambda}\int_{\R^{2d}}\frac{e^{-\rho|y|}}{|y|^{d-1}}\frac{e^{-\rho|z|}}{|z|^{d-1}}|R|(x-\frac{y-z}{\sqrt{\lambda}})dydz.$$ From \eqref{eq:Es3terms}, we have
\begin{equation}
\begin{aligned}
&\E\{|u_\eps(t,x)-u_{hom}(t,x)|\}\\
\les &\left(\sqrt{\lambda\langle \Phi_\lambda,\Phi_\lambda\rangle}+|\sigma_\lambda^2-\sigma^2|+\eps\sqrt{\int_{\R^d}\frac{F_{\lambda,\rho}(x)+F_{\lambda,\rho}^2(x)}{|x|^{d-2}}dx}\right)(1+t),
\end{aligned}
\end{equation}
then Theorem \ref{thm:THgaussian} comes from Lemma \ref{lem:gaussianConvolution} and Proposition \ref{prop:remainder1}, \ref{prop:remainder2}.

\section{Proof of convergence to SPDE}
\label{sec:spde}

From the proof of Theorem \ref{thm:THhomo}, we see that the key assumption for homogenization to occur besides stationarity and ergodicity is the integrability of $\hat{R}(\xi)|\xi|^{-2}$. In other words, $R(x)$ has to decays faster than $|x|^{-2}$ at infinity. In this section, we go beyond Assumption \ref{ass:finiteVar} by assuming $R(x)$ decays sufficiently slowly, and prove the transition to stochasticity from homogenization.

First, we recall that the $n-$th order Hermite polynomial is defined as \begin{equation}
H_n(x)=(-1)^n\exp(\frac{x^2}{2})\frac{d^n}{dx^n}\exp(-\frac{x^2}{2}),
\end{equation} and it has the property that
\begin{equation}
\E\{H_m(X)H_n(Y)\}=\left\{
\begin{array}{ll}
n! (\E\{XY\})^n & m=n,\\
0 & m\neq n,
\end{array}\right.
\end{equation}
if $X,Y\sim N(0,1)$ and are jointly Gaussian.

Under Assumption \ref{ass:longRangeGauss}, $g(x)$ is a stationary Gaussian field with zero mean and unit variance, so we can expand $V$ in Hermite polynomials \cite[Section 3]{taqqu1975weak}:
\begin{equation}
V(x)=\Phi(g(x))=\sum_{n=0}^\infty \frac{V_n}{n!}H_n(g(x)),
\end{equation}
where $V_n=\E\{H_n(g(x))\Phi(g(x))\}$. By the assumption $V_0=0, V_1\neq 0$, we have
\begin{equation}
\begin{aligned}
R(x)=\E\{V(0)V(x)\}=&\E\{\Phi(g(0))\Phi(g(x))\}=\sum_{n=0}^\infty\frac{V_n^2}{(n!)^2}\E\{H_n(g(0))H_n(g(x))\}\\
=&\sum_{n=0}^\infty \frac{V_n^2}{n!}R_g(x)^n=V_1^2 R_g(x)+\sum_{n=2}^\infty \frac{V_n^2}{n!}R_g(x)^n.
\end{aligned}
\end{equation}
Since $\sum_{n=0}^\infty \frac{V_n^2}{n!}<\infty$, $R(x)\sim V_1^2R_g(x)$ as $|x|\to \infty$. In addition, $R_g(x)\sim c_d\prod_{i=1}^d |x_i|^{-\alpha_i}$, which leads to $R(x)\sim V_1^2c_d\prod_{i=1}^d |x_i|^{-\alpha_i}$ as $\min_{i=1,...,d}|x_i|\to \infty$. 

The assumption of $V_1\neq 0$ is crucial for the appearance of Gaussian noise in the limiting equation, and it turns out that by this assumption we can reduce the possibly non-Gaussian case to Gaussian case, namely $V(x)=g(x)$, so conditioning on $B$, $X_\eps(t):=\eps^{-\alpha/2}\int_0^t V(B_s/\eps)ds$ is Gaussian, and we can prove its weak convergence by proving convergence of the conditional mean and variance. Before that, following \cite{hu2011feynman} we define the solution to the limiting SPDE \eqref{eq:limitSPDE}.

\subsection{Limiting SPDE}
We first define the formally-written random variable $\int_0^t\dot{W}(B_s)ds=\int_0^t\int_{\R^d} \delta(x-B_s)W(dx)ds$, where $W(dx)$ is the generalized Gaussian random field independent from Brownian motion $B_t$. We use $\E$ to denote the expectation with respect to $W(dx)$, and assume that the covariance function $\E\{ W(dx)W(dy)\}=\prod_{i=1}^d |x_i-y_i|^{-\alpha_i}dxdy$. For a construction of such generalized Gaussian random field, we refer to e.g. \cite[Section 2]{hu2011feynman}. 

For Brownian motion $B$, we use $B_i(s)$ to denote its $i-$th component. Later below we shall consider a collection of several vector-valued Brownian motions. The $j-$th element of that collection will be denoted by $B^j$, and its value at time $s$ by $B^j_s$, while the value at time $s$ of its $k-$th coordinate would be $B^j_k(s)$.

\begin{proposition}
Assume $\alpha_i\in (0,1), i=1,\ldots,d$ and $\sum_{i=1}^d \alpha_i<2$ and define $Y_\eps(t)=\int_0^t \int_{\R^d}q_\eps(x-B_s)W(dx)ds$, where $q_\eps$ is the density of $N(0,\eps)$. Then $Y_\eps(t)$ converges in $L^2$ as $\eps \to 0$
to some random variable $Y(t)$, denoted as $$Y(t)=\int_0^t\dot{W}(B_s)ds=\int_0^t\int_{\R^d} \delta(x-B_s)W(dx)ds.$$ When conditioning on $B$, then $Y_t$ is a Gaussian random variable with zero mean
and variance
\begin{equation}
\E\{Y(t)^2\}=\int_0^t \int_0^t \frac{1}{\prod_{i=1}^d|B_i(s)-B_i(u)|^{\alpha_i}}dsdu.
\label{eq:varianceLongRange}
\end{equation}
\label{prop:integralNoise}
\end{proposition}

\begin{proof}
We first point out that the RHS of \eqref{eq:varianceLongRange} is almost surely finite, and this comes from the fact that $\alpha_i\in (0,1)$ and $ \sum_{i=1}^d \alpha_i<2$ and
\begin{equation}
\E_B\{\int_0^t \int_0^t \frac{1}{\prod_{i=1}^d|B_i(s)-B_i(u)|^{\alpha_i}}dsdu\}=\int_0^t\int_0^t\frac{1}{|s-u|^{\sum_{i=1}^d \alpha_i/2}}dsdu\prod_{i=1}^d\int_{\R}|x|^{-\alpha_i}q_1(x)dx.
\end{equation}

Secondly, we calculate
\begin{equation}
\E\E_B \{Y_\eps^2(t)\}=\int_0^t\int_0^t\int_{\R^{2d}}\E_B\{q_\eps(x-B_s)q_\eps(y-B_u)\}\frac{1}{\prod_{i=1}^d |x_i-y_i|^{\alpha_i}}dxdydsdu.
\end{equation}
By Lemma \ref{lem:convInalpha} we obtain $\int_{\R^{2d}}q_\eps(x-B_s)q_\eps(y-B_u)\frac{1}{\prod_{i=1}^d |x_i-y_i|^{\alpha_i}}dxdy \to \frac{1}{\prod_{i=1}^d |B_i(s)-B_i(u)|^{\alpha_i}}$ as $\eps \to 0$.
By Lemma \ref{lem:fromJiansong} and the dominated convergence theorem, we have the convergence \begin{equation}
\E\E_B\{Y_\eps^2(t)\}\to \int_0^t\int_0^t \E_B\{\frac{1}{\prod_{i=1}^d |B_i(s)-B_i(u)|^{\alpha_i}}\}dsdu.
\end{equation}

Similarly, we can show 
\begin{equation}
\E\E_B\{Y_{\eps_1}(t)Y_{\eps_2}(t)\}\to \int_0^t\int_0^t \E_B\{\frac{1}{\prod_{i=1}^d |B_i(s)-B_i(u)|^{\alpha_i}}\}dsdu
\end{equation}
 as $\eps_1,\eps_2\to 0$. Thus, we have shown that $\{Y_\eps(t)\}$ is a Cauchy sequence in $L^2$, since $$\lim_{\eps_1,\eps_2\to 0}\E\E_B\{(Y_{\eps_1}(t)-Y_{\eps_2}(t))^2\}=0.$$ The limit is then denoted as $Y(t)=\int_0^t\dot{W}(B_s)ds=\int_0^t\int_{\R^d} \delta(x-B_s)W(dx)ds$.

Next, we consider the conditional distribution. Since $Y_\eps(t)\to Y(t)$ in $L^2$, there exists a subsequence $\eps_k$ such that $Y_{\eps_k}(t)\to Y(t)$ almost surely. Note that $W(dx)$ and $B_t$ are independent, so the probability space is the product space. Then we know that conditioning on the Brownian motion, $Y_{\eps_k}(t)\to Y(t)$ almost surely as $k\to \infty$, and this leads to convergence in distribution. Given $B$, $Y_\eps(t)$ is Gaussian with variance \begin{equation}
\begin{aligned}
\E \{Y_\eps^2(t)\}=&\int_0^t\int_0^t\int_{\R^{2d}}q_\eps(x-B_s)q_\eps(y-B_u)\frac{1}{\prod_{i=1}^d |x_i-y_i|^{\alpha_i}}dxdydsdu\\
\to &\int_0^t\int_0^t\frac{1}{\prod_{i=1}^d |B_i(s)-B_i(u)|^{\alpha_i}}dsdu.
\end{aligned}
\end{equation}
  The proof is complete.
\end{proof}

\begin{remark}
If we define $Y^i(t)=\int_0^t \int_{\R^d}\delta(x-B_s^i)W(dx)ds$ for independent Brownian motions $B^1,B^2$, the same proof implies that $Y^1(t),Y^2(t)$ are jointly Gaussian with covariance function given by $\E\{Y^1(t)Y^2(t)\}=\int_0^t\int_0^t \prod_{i=1}^d |B_i^1(s)-B_i^2(u)|^{-\alpha_i}dsdu$ when conditioning on $B^1,B^2$.
\label{remark:covariance}
\end{remark}

\begin{remark}
By the same discussion as in Proposition \ref{prop:integralNoise}, we can define random variable $\int_0^t \int_{\R^d}\delta(y-x-B_s)W(dy)ds$ as the $L^2$ limit of $\int_0^t \int_{\R^d}q_\eps(y-x-B_s)W(dy)ds$ for any $x\in \R^d$. It is straightforward to check that the joint distribution of 
$\int_0^t \int_{\R^d}\delta(y-x-B_s^1)W(dy)ds,\ldots,\int_0^t \int_{\R^d}\delta(y-x-B_s^N)W(dy)ds$ does not depend on $x$, where $B^i,i=1,\ldots,N$ are independent Brownian motions.
\label{remark:diffStart}
\end{remark}

With the random variables $\int_0^t \int_{\R^d}\delta(y-x-B_s)W(dy)ds$ for any $x\in \R^d$, the solution to the SPDE \begin{equation}
\partial_t u=\frac{1}{2}\Delta u+i\dot{W}u
\label{eq:defSPDE}
\end{equation}
with initial condition $u(0,x)=f(x)$ is formally written by Feynman-Kac formula as 
\begin{equation}
u(t,x)=\E_B\{f(x+B_t)\exp(i\int_0^t\int_{\R^d}\delta(y-x-B_s)W(dy)ds)\}.
\end{equation}
We point that the $u(t,x)$ defined as above coincides with the usual definition of weak solution to SPDE \eqref{eq:defSPDE}:
\begin{definition}
A random field $u(t,x)$ is a weak solution to \eqref{eq:defSPDE} if for any $\mathcal{C}^\infty$ function $\phi$ with compact support we have
\begin{equation}
\int_{\R^d} u(t,x)\phi(x)dx=\int_{\R^d}f(x)\phi(x)dx+\frac{1}{2}\int_0^t\int_{\R^d}u(s,x)\Delta \phi(x)dxds+i\int_{\R^d}\int_0^tu(s,x)\phi(x)dsW(dx),
\end{equation}
where the stochastic integral on the r.h.s. of the above display is understood as a Stratonovich type integral whose meaning is given in Definition 4.1 of \cite{hu2011feynman}.
\end{definition}

\begin{proposition}
If $\alpha_i\in (0,1), i=1,\ldots,d$ and $\sum_{i=1}^d \alpha_i<2$, $u(t,x)$ is a weak solution to \eqref{eq:defSPDE}.
\end{proposition}
The proof is a direct adaption of Theorem 4.3 in \cite{hu2011feynman}, and we do not present it here.

\subsection{Convergence to a stochastic equation: proof of Theorem \ref{thm:THspde}}

First we reduce $V(x)=\Phi(g(x))$ to the Gaussian case by the following lemma:

\begin{lemma}
In the annealed sense, $\mathcal{X}_\eps(t):=\eps^{-\alpha/2}\int_0^t (\Phi(g(B_s/\eps))-V_1g(B_s/\eps))ds\to 0$ in probability as $\eps \to 0$.
\label{lem:reduction}
\end{lemma}

\begin{proof}
Since $\Phi(g)-V_1g=\sum_{n=2}^\infty \frac{V_n}{n!}H_n(g)$ and $\sum_{n=0}^\infty \frac{V_n^2}{n!}<\infty$, we have conditionally upon $B$ that
\begin{equation}
\begin{aligned}
\E\{\mathcal{X}_\eps(t)^2\}=&\frac{1}{\eps^\alpha}\int_0^t\int_0^t \sum_{n=2}^\infty \frac{V_n^2}{n!}R_g(\frac{B_s-B_u}{\eps})^ndsdu\\
\leq &\frac{C}{\eps^\alpha} \int_0^t\int_0^t R_g(\frac{B_s-B_u}{\eps})^2dsdu
\end{aligned}
\end{equation}
for some constant $C$. Since $R_g$ is bounded and satisfies $|R_g(x)|\les \prod_{i=1}^d |x_i|^{-\alpha_i}$, we have
\begin{equation}
\begin{aligned}
\E\{\X_\eps(t)^2\}\leq& C\sup_{|x|\geq M}|R_g(x)|\int_0^t\int_0^t \frac{1}{\prod_{i=1}^d |B_i(s)-B_i(u)|^{\alpha_i}} 1_{|B_s-B_u|>M\eps}dsdu\\
+&\frac{C}{\eps^\alpha}\int_0^t\int_0^t1_{|B_s-B_u|\leq M\eps}dsdu,
\end{aligned}
\end{equation}
which leads to
\begin{equation}
\E\E_B\{\X_\eps(t)^2\}\leq C\sup_{|x|\geq M}|R_g(x)|+\frac{C}{\eps^\alpha}
\int_0^t\int_0^t\E_B\{1_{|B_s-B_u|\leq M\eps}\}dsdu.
 \end{equation}
 By Lemma \ref{lem:BsMinusBu}, first let $\eps\to 0$, then $M\to \infty$, the proof is complete.
\end{proof}

Now we can prove the weak convergence of $X_\eps(t)=\eps^{-\frac{\alpha}{2}}\int_0^tV(B_s/\eps)ds$.
\begin{proposition}
For fixed $t>0$, in the annealed sense $X_\eps(t)\Rightarrow V_1\sqrt{c_d}\int_0^t\int_{\R^d}\delta(x-B_s)W(dx)ds$ as $\eps \to 0$.
\label{prop:conLong}
\end{proposition}

\begin{proof}
By writing $$X_\eps(t)=\frac{1}{\eps^{\alpha/2}}\int_0^t (\Phi(g(\frac{B_s}{\eps}))-V_1g(\frac{B_s}{\eps}))ds+\frac{1}{\eps^{\alpha/2}}\int_0^t V_1g(\frac{B_s}{\eps})ds$$ and applying Lemma \ref{lem:reduction}, we only need to show the weak convergence of $\eps^{-\alpha/2}\int_0^t V_1g(B_s/\eps)ds$.

By conditioning on $B$, we calculate the characteristic function as
\begin{equation}
\E\{\exp(i\theta
\frac{1}{\eps^{\alpha/2}}\int_0^t V_1g(\frac{B_s}{\eps})ds)\}=\exp(-\frac{V_1^2\theta^2}{2\eps^{\alpha}}\int_0^t\int_0^t R_g(\frac{B_s-B_u}{\eps})dsdu).
\end{equation}
Recall that $R_g(x)\sim c_d\prod_{i=1}^d |x_i|^{-\alpha_i}$ as $\min_{i=1,\ldots,d}|x_i|\to \infty$ and $|R_g(x)|\les \prod_{i=1}^d |x_i|^{-\alpha_i}$, we have
 \begin{equation}
 \frac{1}{\eps^\alpha}\int_0^t\int_0^t R_g(\frac{B_s-B_u}{\eps})dsdu\to c_d\int_0^t\int_0^t \frac{1}{\prod_{i=1}^d |B_i(s)-B_i(u)|^{\alpha_i}}dsdu
 \end{equation} almost surely. Now we only need to apply the dominated convergence theorem to derive
\begin{equation}
\begin{aligned}
\E\E_B\{\exp(i\theta
\frac{1}{\eps^{\alpha/2}}\int_0^t V_1g(\frac{B_s}{\eps})ds)\}\to &
\E_B\{\exp(-\frac{1}{2}\theta^2V_1^2c_d\int_0^t\int_0^t \frac{1}{\prod_{i=1}^d |B_i(s)-B_i(u)|^{\alpha_i}}dsdu)\}\\
=& \E\E_B\{\exp(i\theta V_1\sqrt{c_d}\int_0^t\int_{\R^d}\delta(x-B_s)W(dx)ds)\}
\end{aligned}
\end{equation}
as $\eps \to 0$.
\end{proof}

Now we are ready to prove the main theorem.


\begin{proof}[Proof of theorem \ref{thm:THspde}]
For fixed $(t,x)$, we let
\begin{eqnarray}
Z_\eps:=u_\eps(t,x)&=&\E_B\{f(x+B_t)\exp(i\frac{1}{\eps^{\alpha/2}}\int_0^t V(\frac{x+B_s}{\eps})ds)\},\\
Z_0:=u_{spde}(t,x)&=&\E_B\{ f(x+B_t)\exp(i V_1\sqrt{c_d}\int_0^t\int_{\R^d} \delta(y-x-B_s)W(dy)ds)\},
\end{eqnarray}
and claim that $\forall m,n\in \mathbb{N}$, $\E\{ Z_\eps^m \overline{Z_\eps^n}\}\to \E\{ Z_0^m \overline{Z_0^n}\}$. 

Actually, we have
\begin{equation}
\E\{Z_\eps^m \overline{Z_\eps^n}\}=\E\E_B\{\prod_{j=1}^mf(x+B_t^j)\prod_{j=m+1}^{m+n}\overline{f(x+B_t^j)}
\exp(\frac{i}{\eps^{\alpha/2}}\int_0^t(\sum_{j=1}^mV(\frac{x+B_s^j}{\eps})-\sum_{j=m+1}^{m+n}V(\frac{x+B_s^j}{\eps}))ds)\},
\end{equation}
where $B_t^j, j=1,\ldots,N=m+n$ are independent Brownian motions. Since all relevant functions are bounded and continuous, to prove the convergence of $\E\{Z_\eps^m\overline{Z_\eps^n}\}\to \E\{ Z_0^m \overline{Z_0^n}\}$, we only need to prove the annealed weak convergence of
\begin{equation}
\begin{aligned}
W_\eps:=&\sum_{j=1}^N\alpha_j B_t^j+\sum_{j=1}^N\beta_j\frac{1}{\eps^{\alpha/2}}\int_0^tV(\frac{B_s^j}{\eps})ds\\
\Rightarrow &\sum_{j=1}^N\alpha_j B_t^j+V_1\sqrt{c_d}\sum_{j=1}^N\beta_j \int_0^t\int_{\R^d}\delta(y-B_s^j)W(dy)ds
\label{eq:conDisSPDE}
\end{aligned}
\end{equation}
for $\alpha_j,\beta_j\in \R$, where we have used the stationarity of $V(x)$ and Remark \ref{remark:diffStart}. 

Now we write $W_\eps=I_1+I_2+I_3$ with
\begin{eqnarray}
I_1&=&\sum_{j=1}^N\alpha_j B_t^j,\\
I_2&=&\sum_{j=1}^N\beta_j \frac{1}{\eps^{\alpha/2}}\int_0^t V_1g(\frac{B_s^j}{\eps})ds,\\
I_3&=&\sum_{j=1}^N\beta_j \frac{1}{\eps^{\alpha/2}}\int_0^t (\Phi(g(\frac{B_s^j}{\eps}))-V_1g(\frac{B_s^j}{\eps}))ds,
\end{eqnarray}
$I_3\to 0$ in probability by Lemma \ref{lem:reduction}, and for $I_1+I_2$, we calculate
\begin{equation}
\begin{aligned}
&\E\E_B\{\exp(i\theta_1I_2+i\theta_2 I_2)\}\\
=&\E_B\{\exp(i\theta_1\sum_{j=1}^N\alpha_j B_t^j)\exp(-\frac{1}{2}V_1^2\theta_2^2\sum_{i,j=1}^N\beta_i\beta_j\frac{1}{\eps^\alpha}\int_0^t\int_0^tR_g(\frac{B_s^i-B_u^j}{\eps})dsdu)\},
\end{aligned}
\end{equation}
and by the same proof as in Proposition \ref{prop:conLong}, we have
\begin{equation}
\frac{1}{\eps^\alpha}\int_0^t\int_0^tR_g(\frac{B_s^i-B_u^j}{\eps})dsdu\to
\int_0^t\int_0^t \frac{c_d}{\prod_{k=1}^d|B_k^i(s)-B_k^j(u)|^{\alpha_k}}dsdu
\end{equation}
almost surely. Therefore, we see that \begin{equation}
I_1+I_2\Rightarrow
\sum_{j=1}^N\alpha_j B_t^j+V_1\sqrt{c_d}\sum_{j=1}^N\beta_j \int_0^t\int_{\R^d}\delta(y-B_s^j)W(dy)ds
\end{equation}
in distribution in light of Remark \ref{remark:covariance}, so \eqref{eq:conDisSPDE} is proved.

Note that $|Z_\eps|, |Z_0|$ are uniformly bounded, if we let $Z_\eps=Z_{\eps,1}+iZ_{\eps,2}, Z_0=Z_{0,1}+iZ_{0,2}$, the corresponding real and imaginary parts are uniformly bounded as well. From the fact that $\E\{ Z_\eps^m \overline{Z_\eps^n}\}\to \E\{ Z_0^m \overline{Z_0^n}\}$, we know $\forall m,n\in \mathbb{N}$, $\E\{Z_{\eps,1}^mZ_{\eps,2}^n\}\to \E\{Z_{0,1}^mZ_{0,2}^n\}$. So
\begin{equation}
\begin{aligned}
\E\{\exp(i\theta_1Z_{\eps,1}+i\theta_2Z_{\eps,2})\}=&\sum_{k=0}^\infty \frac{1}{k!}\E\{(i\theta_1Z_{\eps,1}+i\theta_2Z_{\eps,2})^k\}\\
\to &\sum_{k=0}^\infty \frac{1}{k!}\E\{(i\theta_1Z_{0,1}+i\theta_2Z_{0,2})^k\}=\E\{\exp(i\theta_1Z_{0,1}+i\theta_2Z_{0,2})\},
\end{aligned}
\end{equation}
which completes the proof.
\end{proof}


\section*{Acknowledgment}  
We would like to thank the anonymous referees for their careful reading of the manuscript and many helpful suggestions and comments. This paper was partially funded by AFOSR Grant NSSEFF- FA9550-10-1-0194 and NSF grant DMS-1108608.


\appendix
\section{Technical lemmas}

\begin{proposition}
Consider the equation $\partial_t u=\frac{1}{2}\Delta u+iV(x)u$
 with initial condition $u(0,x)=f(x)\in \C_b(\R^d)$. Let us define $u(t,x)=\E_B\{ f(x+B_t)\exp(i\int_0^t V(x+B_s)ds)\}$. If $V$ has locally bounded sample path almost surely, we have for any $\varphi \in \C_c^\infty(\R^d)$,
\begin{equation}
\int_{\R^d} \!\!\!   u(t,x)\varphi(x)dx=\int_{\R^d} \!\!\!   f(x)\varphi(x)dx+\int_0^t\int_{\R^d} \!\!\!   u(s,x)\frac{1}{2}\Delta \varphi(x)dxds+i\int_0^t \int_{\R^d} \!\!\! u(s,x)V(x)\varphi(x)dxds,
\label{eq:weaksolution}
\end{equation}
i.e., the Feynman-Kac solution $u(t,x)$ is a weak solution almost surely.
\label{prop:wkSolu}
\end{proposition}

\begin{proof}
Fixing any $\delta,M>0$, define
\begin{equation}
V_{\delta,M}(x)=\int_{\R^d} \phi_\delta(x-y)V(y)1_{|y|<M}dy,
\end{equation}
where $\phi_\delta$ is a family of compactly supported mollifier. Fixing the realization, since $V(y)1_{|y|<M}$ is bounded, $V_{\delta,M}$ is bounded, and we have $V_{\delta,M}(x)\to V(x)1_{|x|<M}$ almost everywhere as $\delta \to 0$. In addition, $V_{\delta,M}$ is smooth, so for the equation $\partial_t u_{\delta, M}=\frac{1}{2}\Delta u_{\delta,M} +iV_{\delta,M}u_{\delta,M}$ with initial condition $u_{\delta,M}(0,x)=f(x)$, we have its classical solution given by the Feynman-Kac formula
\begin{equation}
u_{\delta,M}(t,x)=\E_B\{f(x+B_t)\exp(i\int_0^t V_{\delta,M}(x+B_s)ds)\},
\end{equation}
and if we first let $\delta\to 0$, then $M\to\infty$,  $u_{\delta,M}(t,x)\to u(t,x)$ by the dominated convergence theorem. Since $u_{\delta,M}$ is also a weak solution, we have
\begin{equation}\begin{array}{rcl}
\displaystyle \int_{\R^d} u_{\delta,M}(t,x)\varphi(x)dx&=&\displaystyle \int_{\R^d} f(x)\varphi(x)dx+ \displaystyle \int_0^t\int_{\R^d} u_{\delta,M}(s,x)\frac{1}{2}\Delta \varphi(x)dxds\\&&\displaystyle+i\int_0^t \int_{\R^d} u_{\delta,M}(s,x)V_{\delta,M}(x)\varphi(x)dxds.\end{array}
\end{equation}
Let $\delta \to 0, M\to \infty$, we complete the proof.
\end{proof}

\begin{proposition}
If $M_t$ is a continuous martingale and $W_t$ is a standard Brownian motion, then
\begin{equation}
d_{1,k}(M_1,W_1)\leq (1\vee k)\E\{ |\langle M\rangle_1-1|\},
\end{equation}
with the distance $d_{1,k}$ defined as
\begin{equation}
d_{1,k}(X,Y)=\sup \{|\E\{f(X)-f(Y)\}|: f\in C_b^2(\R), \|f'\|_\infty\leq 1, \|f''\|_\infty\leq k\}.
\end{equation}
\end{proposition}

\begin{proof}
Since $M_t$ is continuous, the quadratic variation process $\langle M\rangle_t$ is continuous as well. We define
\begin{equation}
\tau=\sup\{t\in[0,1]: \langle M\rangle_t\leq 1\},
\end{equation}
and it is clear that $\tau$ is a stopping time. We construct $\tilde{M}_t$ on $[0,2]$ as
\begin{equation}
\tilde{M}_t=
 \left\{
\begin{array}{ll}
M_t & t\in [0,\tau],\\
M_\tau& t\in (\tau,1],\\
M_\tau+b_{t-1} & t\in (1,2-\langle M\rangle_\tau],\\
M_\tau+b_{1-\langle M\rangle_\tau} & t\in (2-\langle M\rangle_\tau,2],
\end{array} \right.
\end{equation}
where $b$ is an independent Brownian motion.

Clearly $\tilde{M}_t$ is a continuous martingale and $\langle\tilde{M}\rangle_2=1$, so $\tilde{M}_2\sim N(0,1)$. Therefore, $d_{1,k}(M_1,W_1)=d_{1,k}(M_1,\tilde{M}_2)$ and we have
\begin{equation}
d_{1,k}(M_1,\tilde{M}_2)\leq d_{1,k}(M_1,M_\tau)+d_{1,k}(M_\tau,\tilde{M}_2).
\end{equation}
For the first term, if $\|f''\|_\infty\leq k$,
\begin{equation}
|\E \{f(M_1)\}-\E \{f(M_\tau)\}-\E\{(M_1-M_\tau)f'(M_\tau)\}|\leq \frac{k}{2}\E\{(M_1-M_\tau)^2\}.
\end{equation}
Note $LHS=|\E \{f(M_1)\}-\E \{f(M_\tau)\}|$ because $\E\{ \E\{(M_1-M_\tau)f'(M_\tau)|\mathcal{F}_\tau\}\}=0$, and $\E\{(M_1-M_\tau)^2\}=\E\{ \langle M\rangle_1-\langle M\rangle_\tau\}\leq \E\{|\langle M\rangle_1-1|\}$.

For the second term, we have $\tilde{M}_2=M_\tau+b_{1-\langle M\rangle_\tau}$. So similarly
\begin{equation}
|\E \{f(\tilde{M}_2)\}-\E \{f(M_\tau)\}-\E\{ b_{1-\langle M\rangle_\tau}f'(M_\tau)\}|\leq \frac{k}{2}\E \{b_{1-\langle M\rangle_\tau}^2\}.
\end{equation}
$LHS=|\E \{f(\tilde{M}_2)\}-\E \{f(M_\tau)\}|$ since $b$ is independent from $M$, and $$RHS= \frac{k}{2}\E\{1-\langle M\rangle_\tau\}\leq \frac{k}{2}\E\{ |1-\langle M\rangle_1|\}.$$

To summarize, we have $d_{1,k}(M_1,W_1)\leq k \E\{|1-\langle M\rangle_1|\}$.
\end{proof}

\begin{lemma}
\begin{equation}
\int_{\R^d}\frac{e^{-\rho|x-y|}}{|x-y|^{d-1}}\frac{e^{-\rho|y|}}{|y|^{d-1}}dy\les e^{-\rho|x|}(1+\frac{1}{|x|^{d-2}}).
\end{equation}
\label{lem:fromWenjia}
\end{lemma}

\begin{proof}
See \cite{bal2011corrector} Lemma A.1.
\end{proof}


The result in Lemma \ref{lem:2ndMoment} is of convolution type. We prove it by the domain decomposition method. Here are some notations appearing in the proof. If we denote $B(z,r)=\{y:|y-z|\leq r\}$,  then $\forall x\in \R^d$, let $\rho=|x|>0$, $A_1=\{z:|z|<|z-x|\}$, $A_2=\{z:|z|\geq |z-x|\}$, and define $(I)=B(0,\rho)\cap A_1$, $(II)=B(x,\rho)\cap A_2$, $(III)=\R^d\setminus((I)\cup(II))$.

$(I),(II),(III)$ appears in the proof of Lemma \ref{lem:2ndMoment}, and we will estimate the integral in each of them respectively. $\Psi$ is assumed to be some positive function such that $\Psi(x)\les 1\wedge |x|^{-\alpha}$ for any $\alpha>0$.

\begin{lemma}
\begin{equation}
\frac{1}{\lambda}\int_{\R^{2d}}\frac{e^{-\rho|y|}}{|y|^{d-1}}\frac{e^{-\rho|z|}}{|z|^{d-1}}\Psi(x-\frac{y-z}{\sqrt{\lambda}})dydz
\les \lambda^{\frac{d}{2}-1}e^{-c\sqrt{\lambda}|x|}+1\wedge\frac{e^{-c\sqrt{\lambda}|x|}}{|x|^{d-2}}+1\wedge\frac{1}{|x|^\beta}
\end{equation}
for some $c>0$ and sufficiently large $\beta>0$.
\label{lem:2ndMoment}
\end{lemma}

\begin{proof}
By Lemma \ref{lem:fromWenjia}, we have
\begin{equation}
\frac{1}{\lambda}\int_{\R^2d}\frac{e^{-\rho|y|}}{|y|^{d-1}}\frac{e^{-\rho|z|}}{|z|^{d-1}}\Psi(x-\frac{y-z}{\sqrt{\lambda}})dydz
\les (i)+(ii),
\end{equation}
where
\begin{eqnarray}
(i)&=&\lambda^{\frac{d}{2}-1}\int_{\R^d}e^{-\rho\sqrt{\lambda}|y|}(1\wedge \frac{1}{|x-y|^\alpha})dy,\\
(ii)&=&\int_{\R^d}e^{-\rho\sqrt{\lambda}|y|}\frac{1}{|y|^{d-2}}(1\wedge \frac{1}{|x-y|^\alpha})dy.
\end{eqnarray}
We have used $\Psi(x)\les 1\wedge \frac{1}{|x|^\alpha}$ for $\alpha$ sufficiently large. $(i),(ii)$ will be estimated separately but in the same way.

First of all, we clearly have that $(i)\les \lambda^{\frac{d}{2}-1}$ and $(ii)\les 1$. Now we assume $|x|\gg1$ and divide $\R^d$ into three parts, $(I),(II),(III)$.

For $(i)$, we have that when $|y-x|\leq 1$, $\int_{|y-x|\leq 1}e^{-\rho\sqrt{\lambda}|y|}dy\les e^{-\rho\sqrt{\lambda}|x|}$. In region $(I)$, we have $|y-x|\geq \frac{|x|}{2}$, so $$\int_{I}e^{-\rho\sqrt{\lambda}|y|}\frac{1}{|x-y|^\alpha}dy\les \frac{1}{|x|^{\alpha-d}}.$$ In region $(II)$, $|y|\geq \frac{|x|}{2}$, so $$\int_{II}1_{|x-y|>1}e^{-\rho\sqrt{\lambda}|y|}\frac{1}{|x-y|^\alpha}dy\les e^{-\rho\sqrt{\lambda}|x|/2}.$$ In region $(III)$, $|x-y|\geq |y|/2$, so $$\int_{III}e^{-\rho\sqrt{\lambda}|y|}\frac{1}{|x-y|^\alpha}dy\les \int_{\R^d}1_{|y|>|x|}\frac{1}{|y|^\alpha}dye^{-\rho\sqrt{\lambda}|x|}\les e^{-\rho\sqrt{\lambda}|x|}.$$ Therefore, in summary, we have
\begin{equation}
\int_{\R^d}e^{-\rho\sqrt{\lambda}|y|}(1\wedge \frac{1}{|x-y|^\alpha})dy\les 1\wedge(e^{-c\sqrt{\lambda}|x|}+\frac{1}{|x|^\beta})
\end{equation}
for $c=\rho/2>0$ and $\beta$ sufficiently large.

For $(ii)$, when $|y-x|\leq 1$, $$\int_{|y-x|\leq 1}e^{-\rho\sqrt{\lambda}|y|}\frac{1}{|y|^{d-2}}\les e^{-\rho\sqrt{\lambda}|x|}\frac{1}{|x|^{d-2}}.$$ In region $(I)$, by a similar discussion, we have $$\int_{(I)}e^{-\rho\sqrt{\lambda}|y|}\frac{1}{|y|^{d-2}}dy\frac{1}{|x|^\alpha}\les \frac{1}{|x|^{\alpha-2}}.$$
In region $(II)$, $e^{-\rho\sqrt{\lambda}|y|}\frac{1}{|y|^{d-2}}\les e^{-\rho\sqrt{\lambda}|x|/2}\frac{1}{|x|^{d-2}}$, so $$\int_{(II)}e^{-\rho\sqrt{\lambda}|y|}\frac{1}{|y|^{d-2}}\frac{1}{|x-y|^\alpha}1_{|x-y|>1}dy\les e^{-\rho\sqrt{\lambda}|x|/2}\frac{1}{|x|^{d-2}}.$$ In region $(III)$,
we have $$\int_{(III)}e^{-\rho\sqrt{\lambda}|y|}\frac{1}{|y|^{d-2}}\frac{1}{|x-y|^\alpha}dy\les e^{-\rho\sqrt{\lambda}|x|}\frac{1}{|x|^{d-2}}.$$

The proof is complete.
\end{proof}

\begin{lemma}
For $$F_{\lambda,\rho}(x)=\frac{1}{\lambda}\int_{\R^{2d}}\frac{e^{-\rho|y|}}{|y|^{d-1}}\frac{e^{-\rho|z|}}{|z|^{d-1}}|R|(x-\frac{y-z}{\sqrt{\lambda}})dydz,$$ and $|R(x)|\les 1\wedge |x|^{-\beta}$ with $\beta\in (2,d)$, we have the following estimates for some $c>0$:
\begin{equation}
\begin{aligned}
F_{\lambda,\rho}(x) \les &\lambda^{\frac{\beta}{2}-1}e^{-c\sqrt{\lambda}|x|}+\frac{1}{\lambda|x|^\beta}\int_0^{\sqrt{\lambda}|x|}e^{-cr}r^{d-1}dr1_{|x|\geq \frac12}+\lambda^{\frac{d}{2}-1}e^{-c\sqrt{\lambda}|x|}|x|^{d-\beta}1_{|x|\geq 1}\\
+&1\wedge\left(\frac{1}{|x|^{\beta-2}}e^{-c\sqrt{\lambda}|x|}+\frac{1}{\lambda|x|^\beta}\int_0^{\sqrt{\lambda}|x|}e^{-cr}rdr+\lambda^{\frac{\beta}{2}-1}\int_{\sqrt{\lambda}|x|}^\infty e^{-cr}r^{1-\beta}dr\right),
\end{aligned}
\end{equation}
and we have
\begin{equation}
\int_{\R^d}\frac{F_{\lambda,\rho}(x)+F_{\lambda,\rho}^2(x)}{|x|^{d-2}}dx\les\left\{
\begin{array}{ll}
\lambda^{\frac{\beta}{2}-2} & \beta<4,\\
\log|\lambda| & \beta=4,\\
1 & \beta>4.
\end{array} \right.
\end{equation}

%
%

\label{lem:gaussianConvolution}
\end{lemma}

\begin{proof}
The proof is similar to that of Lemma \ref{lem:2ndMoment} and \ref{lem:fluctuationError}. The details are not presented here.
\end{proof}

\begin{lemma}
Let $x_i\in \R^d,i=1,\ldots,4$, then under Assumption \ref{ass:mixing}
\begin{equation}
\begin{aligned}
&|\E\{V(x_1)V(x_2)V(x_3)V(x_4)\}-R(x_1-x_2)R(x_3-x_4)|\\
\leq &\Psi(|x_1-x_3|)\Psi(|x_2-x_4|)+\Psi(|x_1-x_4|)\Psi(|x_2-x_3|),
\end{aligned}
\end{equation}
where $\Psi(r)\les 1\wedge r^{-\beta}$ for any $\beta>0$.
\label{lem:34Moment}
\end{lemma}

\begin{proof}
The proof could be found in Lemma 2.3. \cite{hairer2013random}, where $\E\{V^6(x)\}<\infty$ is used.
%
\end{proof}

\begin{lemma}
When $\alpha\in (0,1)$, $\int_{\R^{2}} q_\eps(x)q_\eps(y)\frac{1}{|z+x-y|^\alpha}dxdy \to \frac{1}{|z|^\alpha}$ as $\eps \to 0$ for $z\neq 0$.
\label{lem:convInalpha}
\end{lemma}

\begin{proof}
By change of variables, we write
\begin{equation}
\begin{aligned}
\int_{\R^{2}} q_\eps(x)q_\eps(y)\frac{1}{|z+x-y|^\alpha}dxdy=&\int_{\R^2} q_\eps(w+y-z)q_\eps(y)\frac{1}{|w|^\alpha}dydw\\
=&\left(\int_{|w|<\frac{|z|}{2}}+\int_{|w|>\frac{|z|}{2}}\right)q_\eps(w+y-z)q_\eps(y)\frac{1}{|w|^\alpha}dydw\\
=&(i)+(ii),
\end{aligned}
\end{equation}
and since
\begin{equation}
(ii)=\int_{|\sqrt{\eps}w+z|>\frac{|z|}{2}}q(w+y)q(y)\frac{1}{|\sqrt{\eps}w+z|^\alpha}dydw,
\end{equation}
by the dominated convergence theorem, we have $(ii)\to \frac{1}{|z|^\alpha}$ as $\eps \to 0$. For $(i)$, we write
\begin{equation}
(i)=\left(\int_{|w|<\frac{|z|}{2},|y|>\frac{|z|}{4}}+\int_{|w|<\frac{|z|}{2},|y|<\frac{|z|}{4}}\right)q_\eps(w+y-z)q_\eps(y)\frac{1}{|w|^\alpha}dydw.
\end{equation}
For the first term, use $q_\eps(|z|/4)$ to bound $q_\eps(y)$, then integrate in $y,w$; for the second term, use $q_\eps(|z|/4)$ to bound $q_\eps(w+y-z)$, then integrate in $y,w$. Since $q_\eps(|z|/4)\to 0$ as $\eps \to 0$, we have $(i)\to 0$. The proof is complete.
\end{proof}

\begin{lemma}
Assume $\alpha\in (0,1)$, then $\int_{\R^2} q_{\eps_1}(x_1+y_1)q_{\eps_2}(x_2+y_2)|y_1-y_2|^{-\alpha}dy_1dy_2\leq C|x_1-x_2|^{-\alpha}$ for some uniform constant $C$.
\label{lem:fromJiansong}
\end{lemma}

\begin{proof}
See Lemma A.2. in \cite{hu2011feynman}.
\end{proof}

\begin{lemma}
When $d\geq 3$ and $\alpha\in (0,2)$, \begin{equation}
\lim_{\eps\to 0}\frac{1}{\eps^\alpha}\int_0^t\int_0^t \Pb(|B_s-B_u|\leq \eps)dsdu=0.
\end{equation}
\label{lem:BsMinusBu}
\end{lemma}

\begin{proof}
By explicit calculation, we have
\begin{equation}
\begin{aligned}
&\frac{1}{\eps^\alpha}\int_0^t\int_0^t \Pb(|B_s-B_u|<\eps)dsdu\\
=&\frac{1}{(\pi)^{\frac{d}{2}}\eps^\alpha}\int_0^t \int_{|x|<\eps}\int_{\frac{|x|^2}{2s}}^\infty \lambda^{\frac{d}{2}-2}e^{-\lambda} \frac{1}{|x|^{d-2}}d\lambda dxds\\
=&\frac{1}{(\pi)^{\frac{d}{2}}\eps^\alpha}\int_0^\infty \int_{\R^d}\int_{0}^\infty 1_{|x|<\eps}1_{|x|^2<2\lambda s}1_{s<t}\lambda^{\frac{d}{2}-2}e^{-\lambda} \frac{1}{|x|^{d-2}}d\lambda dxds\\
=&\frac{1}{(\pi)^{\frac{d}{2}}\eps^\alpha}\int_0^\infty d\lambda \int \lambda^{\frac{d}{2}-2}e^{-\lambda}\left(\lambda s1_{\lambda<\frac{\eps^2}{2s}}+\frac{1}{2}\eps^21_{\lambda>\frac{\eps^2}{2s}}\right)1_{s<t}ds\\
=&\frac{1}{(\pi)^{\frac{d}{2}}\eps^\alpha}\int_0^\infty d\lambda  \lambda^{\frac{d}{2}-2}e^{-\lambda}\left(\frac{\lambda t^2}{2}1_{\frac{\eps^2}{2\lambda}>t}+\frac{\eps^2t}{2}1_{\frac{\eps^2}{2\lambda}<t}-\frac{\eps^4}{8\lambda}1_{\frac{\eps^2}{2\lambda}<t}\right)=(i)+(ii)+(iii).
\end{aligned}
\end{equation}
We check that $(i)\sim \eps^{d-\alpha}$, and $(ii)\sim \eps^{2-\alpha}$, $(iii)\sim \eps^{4-\alpha}+\eps^{d-\alpha}$, so the proof is complete.
\end{proof}


\begin{thebibliography}{10}

\bibitem{armstrong2013quantitative}
{\sc S.~N. Armstrong and C.~K. Smart}, {\em Quantitative stochastic
  homogenization of elliptic equations in nondivergence form}, arXiv preprint
  arXiv:1306.5340,  (2013).

\bibitem{bal2008central}
{\sc G.~Bal}, {\em Central limits and homogenization in random media},
  Multiscale Modeling \& Simulation, 7 (2008), pp.~677--702.

\bibitem{bal2009convergence}
\leavevmode\vrule height 2pt depth -1.6pt width 23pt, {\em Convergence to spdes
  in stratonovich form}, Communications in Mathematical Physics, 292 (2009),
  pp.~457--477.

\bibitem{bal2010homogenization}
{\sc G.~Bal}, {\em Homogenization with large spatial random potential},
  Multiscale Modeling \& Simulation, 8 (2010), pp.~1484--1510.

\bibitem{B-AMRX-11}
\leavevmode\vrule height 2pt depth -1.6pt width 23pt, {\em Convergence to
  homogenized or stochastic partial differential equations}, Appl Math Res
  Express, 2011(2) (2011), pp.~215--241.

\bibitem{bal2012corrector}
{\sc G.~Bal, J.~Garnier, Y.~Gu, and W.~Jing}, {\em Corrector theory for
  elliptic equations with oscillatory and random potentials with long range
  correlations}, Asymptot. Anal, 77 (2012), pp.~123--145.

\bibitem{bal2008random}
{\sc G.~Bal, J.~Garnier, S.~Motsch, and V.~Perrier}, {\em Random integrals and
  correctors in homogenization}, Asymptotic Analysis, 59 (2008), pp.~1--26.

\bibitem{bal2011corrector}
{\sc G.~Bal and W.~Jing}, {\em Corrector theory for elliptic equations in
  random media with singular green's function. application to random
  boundaries}, Commun. Math. Sci, 19 (2011), pp.~383--411.

\bibitem{bolthausen1989central}
{\sc E.~Bolthausen}, {\em A central limit theorem for two-dimensional random
  walks in random sceneries}, The Annals of Probability,  (1989), pp.~108--115.

\bibitem{bourgeat1999estimates}
{\sc A.~Bourgeat and A.~Piatnitski}, {\em Estimates in probability of the
  residual between the random and the homogenized solutions of one-dimensional
  second-order operator}, Asymptotic Analysis, 21 (1999), pp.~303--315.

\bibitem{caffarelli2010rates}
{\sc L.~A. Caffarelli and P.~E. Souganidis}, {\em Rates of convergence for the
  homogenization of fully nonlinear uniformly elliptic pde in random media},
  Inventiones mathematicae, 180 (2010), pp.~301--360.

\bibitem{conlon2000homogenization}
{\sc J.~G. Conlon and A.~Naddaf}, {\em On homogenization of elliptic equations
  with random coefficients}, Electron. J. Probab., 5 (2000), pp.~no. 9, 58 pp.
  (electronic).

\bibitem{ethier2009markov}
{\sc S.~N. Ethier and T.~G. Kurtz}, {\em Markov processes: characterization and
  convergence}, vol.~282, Wiley. com, 2009.

\bibitem{figari1982mean}
{\sc R.~Figari, E.~Orlandi, and G.~Papanicolaou}, {\em Mean field and gaussian
  approximation for partial differential equations with random coefficients},
  SIAM Journal on Applied Mathematics, 42 (1982), pp.~1069--1077.

\bibitem{gloria2011optimal}
{\sc A.~Gloria and F.~Otto}, {\em An optimal variance estimate in stochastic
  homogenization of discrete elliptic equations}, The Annals of Probability, 39
  (2011), pp.~779--856.

\bibitem{gu2012random}
{\sc Y.~Gu and G.~Bal}, {\em Random homogenization and convergence to integrals
  with respect to the rosenblatt process}, Journal of Differential Equations,
  253 (2012), pp.~1069--1087.

\bibitem{gu2014invariance}
\leavevmode\vrule height 2pt depth -1.6pt width 23pt, {\em An invariance
  principle for brownian motion in random scenery}, Electron. J. Probab, 19
  (2014), pp.~1--19.

\bibitem{hairer2013random}
{\sc M.~Hairer, E.~Pardoux, and A.~Piatnitski}, {\em Random homogenisation of a
  highly oscillatory singular potential}, Stochastic Partial Differential
  Equations: Analysis and Computations, 1 (2013), pp.~571--605.

\bibitem{hu2011feynman}
{\sc Y.~Hu, D.~Nualart, and J.~Song}, {\em Feynman--kac formula for heat
  equation driven by fractional white noise}, The Annals of Probability, 39
  (2011), pp.~291--326.

\bibitem{kesten1979limit}
{\sc H.~Kesten and F.~Spitzer}, {\em A limit theorem related to a new class of
  self similar processes}, Probability Theory and Related Fields, 50 (1979),
  pp.~5--25.

\bibitem{kipnis1986central}
{\sc C.~Kipnis and S.~Varadhan}, {\em Central limit theorem for additive
  functionals of reversible markov processes and applications to simple
  exclusions}, Communications in Mathematical Physics, 104 (1986), pp.~1--19.

\bibitem{komorowski2012fluctuations}
{\sc T.~Komorowski, C.~Landim, and S.~Olla}, {\em Fluctuations in Markov
  processes: time symmetry and martingale approximation}, vol.~345, Springer,
  2012.

\bibitem{komorowski2010asymptotic}
{\sc T.~Komorowski and E.~Nieznaj}, {\em On the asymptotic behavior of
  solutions of the heat equation with a random, long-range correlated
  potential}, Potential Analysis, 33 (2010), pp.~175--197.

\bibitem{kozlov1979averaging}
{\sc S.~M. Kozlov}, {\em Averaging of random operators}, Matematicheskii
  Sbornik, 151 (1979), pp.~188--202.

\bibitem{lejay2001homogenization}
{\sc A.~Lejay}, {\em Homogenization of divergence-form operators with
  lower-order terms in random media}, Probability theory and related fields,
  120 (2001), pp.~255--276.

\bibitem{marahrens2013annealed}
{\sc D.~Marahrens and F.~Otto}, {\em Annealed estimates on the green's
  function}, arXiv preprint arXiv:1304.4408,  (2013).

\bibitem{mourrat2012kantorovich}
{\sc J.-C. Mourrat}, {\em Kantorovich distance in the martingale clt and
  quantitative homogenization of parabolic equations with random coefficients},
  Probability Theory and Related Fields,  (2012), pp.~1--36.

\bibitem{papanicolaou1979boundary}
{\sc G.~C. Papanicolaou and S.~R.~S. Varadhan}, {\em Boundary value problems
  with rapidly oscillating random coefficients}, in Random fields, Vol. I, II
  (Esztergom, 1979), Colloq. Math. Soc. J{\'a}nos Bolyai, 27, North Holland,
  Amsterdam, New York, 1981, pp.~835--873.

\bibitem{pardoux2006homogenization}
{\sc E.~Pardoux and A.~Piatnitski}, {\em Homogenization of a singular random
  one dimensional pde}, GAKUTO Internat. Ser. Math. Sci. Appl, 24 (2006),
  pp.~291--303.

\bibitem{pardoux2012homogenization}
{\sc {\'E}.~Pardoux and A.~Piatnitski}, {\em Homogenization of a singular
  random one-dimensional pde with time-varying coefficients}, The Annals of
  Probability, 40 (2012), pp.~1316--1356.

\bibitem{remillard1991limit}
{\sc B.~R{\'e}millard and D.~Dawson}, {\em A limit theorem for brownian motion
  in a random scenery}, Canad. Math. Bull, 34 (1991), pp.~385--391.

\bibitem{taqqu1975weak}
{\sc M.~S. Taqqu}, {\em Weak convergence to fractional brownian motion and to
  the rosenblatt process}, Probability Theory and Related Fields, 31 (1975),
  pp.~287--302.

\bibitem{taylor2011partial}
{\sc M.~E. Taylor}, {\em Partial differential equations I: Basic theory},
  Applied Mathematical Sciences,  (2011).

\bibitem{yurinskii1986averaging}
{\sc V.~Yurinskii}, {\em Averaging of symmetric diffusion in random medium},
  Siberian Mathematical Journal, 27 (1986), pp.~603--613.

\bibitem{ZB-CMS-13}
{\sc N.~Zhang and G.~Bal}, {\em {Convergence to SPDE of the Schr{\"o}dinger
  equation with large, random potential}}, To appear in Comm. Math. Sci.,
  (2013).

\bibitem{ZB-SD-13}
\leavevmode\vrule height 2pt depth -1.6pt width 23pt, {\em {Homogenization of a
  Schr\"odinger equation with large, random, potential }}, To appear in
  Stochastics and Dynamics,  (2013).

\end{thebibliography}

\end{document}